\documentclass[12pt]{amsart}
\usepackage{xcolor}
\usepackage{graphicx}
\usepackage{amsmath}
\usepackage{amssymb}
\usepackage{tikz-cd}
\usepackage{quiver}
\usepackage{enumitem}
\usepackage{mathtools}

\usepackage[backend=biber,
style=alphabetic,
sorting=ynt]{biblatex}
\usepackage{lipsum}
\usepackage[colorlinks=true,
  linkcolor=black,
  citecolor=blue,
  urlcolor=blue]{hyperref}
\usepackage[margin=1in]{geometry}
\theoremstyle{definition}
\newtheorem{definition}{Definition}[section]
\makeatletter
  {\begin{definition}[#1]\normalfont\smallskip\hrule\smallskip}%
  {\smallskip\hrule\smallskip\end{definition}}
\makeatother

\theoremstyle{plain}
\newtheorem{theorem}{Theorem}[section]
\newtheorem{lemma}[theorem]{Lemma}
\newtheorem{claim}[theorem]{Claim}
\theoremstyle{definition}
\newtheorem{axiom}{Axiom}[section]
\newtheorem{proposition}{Proposition}[section]

\makeatletter
  {\begin{theorem}[#1]\smallskip\hrule\smallskip}%
  {\smallskip\hrule\smallskip\end{theorem}}

  {\begin{lemma}[#1]\smallskip\hrule\smallskip}%
  {\smallskip\hrule\smallskip\end{lemma}}

  {\begin{claim}[#1]\smallskip\hrule\smallskip}%
  {\smallskip\hrule\smallskip\end{claim}}
\makeatother

\makeatletter
  {\begin{proof}[#1]\smallskip\hrule\smallskip}%
  {\smallskip\hrule\smallskip\end{proof}}
\makeatother

\title[Double Categorical AQFT I]{Double Categorical Approaches to AQFT I: Axiomatic Setup}
\author{Khyathi Komalan}

\address{Department of Physics, Mathematics and Astronomy, California Institute of Technology}
\email{kkomalan@caltech.edu}

\begin{document}

\begin{abstract}
In operator-algebraic AQFT one routinely moves back and forth between two kinds of structure: inclusions of local algebras coming from inclusions of regions, and bimodules/intertwiners that implement the standard $L^2$-based constructions used to compare and compose observables. The obstruction to making this interplay genuinely functorial is that there are two independent compositions (restriction along inclusions and fusion/transport along bimodules) and they must be compatible on commuting spacetime diagrams, which is exactly the situation a double category is designed to encode. Part I resolves this by building a spacetime double category and a von Neumann algebra double category inspired by previous work by Orendain, and by packaging an AQFT input as a pseudo double functor whose vertical part is the Haag-Kastler net and whose squares record the required compatibilities in a well-typed way forced by commutativity. We formulate the Haag-Kastler axioms in this setup, establish the coherence needed for the construction, and work out representative examples.
\end{abstract}

\maketitle

\tableofcontents

\section{Introduction}

\textbf{Context.} Dr. Skeptical, a respected physicist with a habit of questioning formalism, runs two laboratories in different locations. Each laboratory runs its own experiment, is switched on only for a finite duration, and records a definite outcome. The laboratories are positioned far enough apart that, during the experiment, no physical influence can travel from one laboratory to the other.

Although the laboratories are physically separated, the experiments are not independent in design. A single source prepares a correlated quantum system and sends one component to each laboratory. Each laboratory performs a fixed measurement on the system it receives. When Dr. Skeptical examines the data from either laboratory in isolation, the results appear unremarkable, as the outcomes seem random.

She then assigns one assistant to each laboratory. Each assistant is instructed to analyse only the measurements performed locally and to describe the experiment using whatever theoretical tools they find most appropriate. When questioned separately, Dr. Skeptical finds that the assistants give complete and internally consistent accounts of their findings, from the observables available in the laboratory to the outcomes that occurred. Both assistants agree that the existence of the other laboratory did not affect their results.

However, when Dr. Skeptical pieces together the information provided by both assistants, she is surprised. Strong correlations emerge between the outcomes. While the assistants agree on essentially all local facts, they disagree on how to formalise the experiment. Dr. Skeptical is especially frustrated that each description relies on global structures that are not accessible from within a single laboratory.

Unhappy with this state of affairs, Dr. Skeptical makes a more drastic move. She replaces her assistants with two identical copies of herself (the feasibility of cloning is not the subject of this paper), stationing one at each laboratory and restricting each to analysing only the observables measurable in that region. The result, however, is unchanged. Each copy produces a complete and coherent description of the local experiment, yet there is still no preferred or observer-independent way to combine these descriptions into a single account that respects both the quantum correlations and the causal separation of the laboratories.

At this point, Dr. Skeptical realises that the issue is not the competence of her assistants, or even her own, but rather pertains to the nature of quantum theory. Despite experiments being conducted in causally separated regions, standard formulations of quantum theory tend to place all observables within a single global structure. In non-interacting theories this global viewpoint may be harmless. However, in the presence of interactions or entanglement, particle interpretations become ambiguous, different choices of fields or Lagrangians can describe the same observable behaviour, and the physical content of the theory is no longer tied to any particular microscopic description.

What remains stable across all these descriptions is not the choice of fields nor the global state, but the structure of measurements available in each region and the relations between them. In Dr. Skeptical’s situation, a formulation of quantum field theory built from region-based collections of observables, together with principles governing their inclusion, independence, and mutual consistency, would resolve the difficulty.

The resolution of Dr. Skeptical’s predicament is provided by Algebraic Quantum Field Theory (AQFT), which reformulates quantum field theory in terms of observables localized in spacetime regions rather than globally defined fields or states. The general idea is to encode all physically meaningful information directly at the level of local measurements and their relations, independent of the specific field-theoretic representation chosen \cite{buchholz2024algebraicquantumfieldtheory}.

More concretely, in AQFT, one associates to each bounded region $O$ of Minkowski spacetime an algebra $\mathcal{A}(O)$ of observables localized in that region. The assignment $O \mapsto \mathcal{A}(O)$ is organized as a net over spacetime. If $O_{1}\subset O_{2}$, then observables localized in $O_{1}$ are included in those localized in $O_{2}$, expressing a property known as isotony. Relativistic locality is encoded algebraically by requiring that observables associated with spacelike separated regions commute. In this way, locality is built directly into the observable structure, while correlations between distant regions arise through states rather than nonlocal observables \cite{haag2012local}.

A natural consequence of this setup is that no preferred global description is required. Distinct choices of fields or Lagrangians may generate the same net of local observable algebras and are therefore regarded as physically equivalent. From a more mathematical perspective, one may work with a net of local algebras satisfying a collection of physically motivated principles — known as the Haag-Kastler axioms — rather than with any particular calculational scheme \cite{halvorson2006algebraicquantumfieldtheory}.

Where does category theory enter this picture? Categorically, an AQFT may be described as a covariant functor from a category of spacetime regions — typically open, bounded regions of Minkowski spacetime with inclusions as morphisms — into a category of operator algebras. In this formulation, the net structure $O \mapsto \mathcal{A}(O)$ is encoded functorially, while the Haag-Kastler axioms are imposed as additional conditions on this assignment, specifying which functors are physically admissible \cite{halvorson2006algebraicquantumfieldtheory, haag2012local}.

While this $1$-functorial viewpoint provides a clear framework, it does not exhaust the structural content of an AQFT. In particular, locality, covariance, and time-slice properties are enforced by axioms constraining the functor rather than being determined by functoriality itself, and equivalences between quantum field theories are not intrinsic to the categorical structure. As emphasized in recent work, this limits the ability of the $1$-categorical formulation to encode local-to-global behavior, causal compatibility, and homotopical notions of equivalence in a structural way, motivating the consideration of richer categorical frameworks \cite{carmona2023newmodelstructuresalgebraic}.

\textbf{Double Categories.} Double categories appear in the work of Charles Ehresmann in the early 1960s, in the context of his investigations into structured and internal categories \cite{ehresmann1963categoriesdoubles}\cite{ehresmann1963categoriesstructurees}. In these works, Ehresmann introduces double categories as instances of structured categories equipped with composition laws, allowing morphisms to be composed in two distinct directions. The basic data of a double category involves objects, two classes of morphisms (horizontal and vertical), and squares encoding their mutual compatibility.

Subsequent work developed Ehresmann's framework by clarifying the internal structure and categorical properties of double categories. Authors focused on lifting constructions familiar from $1$-category theory to non-trivial double categorical formulations that don't involve trivialising either horizontal or vertical morphisms. Notable among these works is a series of papers by Marco Grandis and Robert Paré formalizing limits, adjoints, and Kan extensions \cite{grandis1999limitsdoublecategories,
      grandis2004adjointsdoublecategories,
      grandis2008kanextensionsdoublecategories}.

Other formulations of double categories have been considered, in regard to the strictness with which composition and unit laws are imposed. In some settings it is more natural for these laws to hold strictly in both directions, while in others, some more generalisation is required, where these properties are only present up to isomorphism. For this reason, notions such as strict, weak, and pseudo double categories have been introduced.

We recall the definition of a double category \cite{nlab:double_category}. (We begin with the strict/internal definition for concreteness; the pseudo variant used in the construction is recalled in Section~\ref{sec:double}.)

\begin{definition}[Double Category]
A double category $\mathbb{D}$ consists of two categories $\mathbb{D}_{0}$ and $\mathbb{D}_{1}$, together with source and target functors
$$
s,t:\mathbb{D}_{1}\rightrightarrows \mathbb{D}_{0}.
$$
The objects of $\mathbb{D}_{0}$ are the objects of $\mathbb{D}$, and its morphisms are the vertical morphisms.
The objects of $\mathbb{D}_{1}$ are the horizontal morphisms, while its morphisms are the $2$-cells (or squares).

Pictorially, this looks like:
\[
\begin{tikzcd}
  {X_0} & {X_1} \\
  {Y_0} & {Y_1}
  \arrow[""{name=0, anchor=center, inner sep=0}, "{h_0}", from=1-1, to=1-2]
  \arrow["{v_0}"', from=1-1, to=2-1]
  \arrow["{v_1}", from=1-2, to=2-2]
  \arrow[""{name=1, anchor=center, inner sep=0}, "{h_1}"', from=2-1, to=2-2]
  \arrow["\phi", Rightarrow, from=0, to=1]
\end{tikzcd}
\]
for objects $X_{i},Y_{i}$, horizontal morphisms $h_{i}$, vertical morphisms $v_{i}$, and a $2$-cell $\phi$.

In addition, there is an identity functor
$$
i:\mathbb{D}_{0}\to \mathbb{D}_{1},
$$
and a horizontal composition functor
$$
\circ_{h}:\mathbb{D}_{1}\times_{\mathbb{D}_{0}}\mathbb{D}_{1}\to \mathbb{D}_{1},
$$
such that these data define a category internal to $\mathbf{Cat}$.

In this spirit, a double functor $\mathbb{F}:\mathbb{D}\to \mathbb{E}$ between double categories
consists of functorial data mapping objects, vertical morphisms, horizontal morphisms,
and $2$-cells of $\mathbb{D}$ to those of $\mathbb{E}$, compatibly with source and target maps,
identities, and horizontal and vertical composition \cite{myersCategoricalSystemsTheory}.
\end{definition}

As briefly mentioned earlier, AQFT admits a $1$-categorical formulation, given by a functor from a category of spacetime regions with inclusions to a category of operator algebras.
We return to this formulation in Section~\ref{sec:standard}, to explore it in more detail. However, the $1$-categorical formulation deals with only inclusions of regions and the induced algebra homomorphisms. Other physically relevant relations — such as causal compatibility between regions or correlation data relating distinct regions — are not represented as morphisms, and their interaction with inclusions is imposed externally via axioms.

To get an intuitive understanding of the idea behind our double categorical AQFT construction, let us go back to the Dr. Skeptical anecdote.

Let $U_{0}\subset U_{1}$ be regions in laboratory one and $V_{0}\subset V_{1}$ be regions in laboratory two. Consider a square in which the vertical arrows are inclusions within each laboratory, while the horizontal arrows record the relevant ``between-laboratories'' spacetime data:
\[
\begin{tikzcd}
  {U_0} & {V_0} \\
  {U_1} & {V_1}
  \arrow[""{name=top, anchor=center, inner sep=0}, "{\text{causal map}}", from=1-1, to=1-2]
  \arrow["{\text{inclusion}}"', from=1-1, to=2-1]
  \arrow["{\text{inclusion}}", from=1-2, to=2-2]
  \arrow[""{name=bot, anchor=center, inner sep=0}, "{\text{causal map}}"', from=2-1, to=2-2]
  \arrow["\phi", Rightarrow, from=top, to=bot]
\end{tikzcd}
\]
Here, $\phi$ represents compatibility between passing to larger regions (vertically) and the chosen spacetime relation between the laboratories (horizontally). In the $1$-categorical setup, the domain category representing spacetime has only inclusions as morphisms, whereas a double categorical perspective allows one to track, simultaneously, inclusions and a second class of physically meaningful morphisms together with their compatibilities.

Similarly, let us think about the corresponding square of observable data. Since we are only sketching the idea here, with a more rigorous treatment deferred to Section~\ref{sec:double}, we denote the algebras associated to regions in the first laboratory by $\mathcal{A}(U_{i})$ and those associated to regions in the second laboratory by $\mathcal{A}(V_{i})$:
\[
\begin{tikzcd}
  {\mathcal{A}(U_0)} & {\mathcal{A}(V_0)} \\
  {\mathcal{A}(U_1)} & {\mathcal{A}(V_1)}
  \arrow[""{name=top, anchor=center, inner sep=0}, "{h_0}", from=1-1, to=1-2]
  \arrow["{v_0}"', from=1-1, to=2-1]
  \arrow["{v_1}", from=1-2, to=2-2]
  \arrow[""{name=bot, anchor=center, inner sep=0}, "{h_1}"', from=2-1, to=2-2]
  \arrow["\phi", Rightarrow, from=top, to=bot]
\end{tikzcd}
\]
We concretely define what $h_i$ and $v_i$ are later, but their roles can already be explained intuitively.
The vertical morphisms $v_i$ are the $\ast$-homomorphisms induced by inclusions of spacetime regions within each laboratory, encoding isotony: observables localized in a smaller region embed into those localized in a larger one. These are precisely the algebraic maps that appear in the standard AQFT formulation.

The horizontal morphisms $h_i$, by contrast, do not arise from inclusions. Instead, they represent Hilbert bimodules relating the algebras associated to regions in different laboratories. Physically, these bimodules encode correlation data or shared state-space structure between the two laboratories without identifying their observables. The square $\phi$ then expresses the compatibility between isotony within each laboratory and relating the two laboratories via such bimodules.

In our setup, we define AQFT to be a double functor between the double category of spacetime and the double category of observables.

In Dr. Skeptical’s situation, the square above formalizes the fact that local descriptions within each laboratory remain complete and consistent, while correlations between laboratories persist and must be accounted for in a way that respects both inclusion and causal separation. This compatibility is not visible in the $1$-categorical formulation, but becomes explicit in the double categorical setting.

\textbf{Scope.} The aim of this paper is to formulate algebraic quantum field theory within a double categorical framework that separates inclusions of spacetime regions from correlation data relating observables. We construct a double category of spacetime regions and a corresponding double category of operator-algebraic data, and define a double functor between them that refines the standard $1$-categorical formulation of AQFT.

Concretely, we construct double categories $\mathbf{Mink}(M)$ and $\mathbf{vNA}$ suited to AQFT, and we define the central object of the paper to be a double functor
$$
\mathbb{F}:\mathbf{Mink}(M)\to \mathbf{vNA}.
$$
The double category $\mathbf{Mink}(M)$ has causally convex regions as objects, inclusions as vertical morphisms, causal embeddings as horizontal morphisms, and commuting squares as $2$-cells. The double category $\mathbf{vNA}$ has von Neumann algebras as objects, normal unital $\ast$-homomorphisms as vertical morphisms, Hilbert bimodules as horizontal morphisms (composed by Connes fusion), and intertwiners as $2$-cells. The functor $\mathbb{F}$ sends a region $U$ to its algebra $\mathcal{A}(U)$, sends inclusions to the isotony embeddings, and sends horizontal spacetime data to bimodules encoding correlation or transport. Restricting $\mathbb{F}$ to the vertical direction recovers an ordinary AQFT net, while the square-level compatibility makes explicit how isotony interacts with the horizontal correspondence calculus in the target.

We show how the Haag-Kastler axioms can be expressed within this framework, distinguishing those aspects that arise from the categorical structure itself from those imposed as additional conditions. Particular emphasis is placed on making explicit the compatibility between isotony and correlation data, encoded by squares in the double category, which is not visible in the $1$-categorical setting.

This paper is devoted to the foundational and structural aspects of the double categorical formulation and to representative examples. Subsequent parts of this work develop applications of the framework: Part II studies its interaction with Type I and Type II von Neumann algebras, while Part III addresses Type III von Neumann algebras and introduces an operadic refinement of the construction, providing a natural setting for Tomita-Takesaki modular theory.

\textbf{Acknowledgements.} The author is grateful to Topos Institute Oxford for welcoming them as an informal research affiliate and for providing a temporary workspace where these ideas were developed, as well as for the opportunity to present related work in the Topos Oxford Seminar.

\textbf{Funding.} This work was supported by a generous grant from Emergent Ventures.

\section{Categorical Treatments of AQFT}\label{sec:standard}

This section is a quick tour of other categorical approaches to AQFT. Nearly all of them are $1$-categorical at the basic level: one indexes observables by a category of regions or spacetimes, and then imposes the physical content by extra conditions.

\textbf{Local nets on a fixed spacetime.} Fix a spacetime background $M$. The algebraic approach starts by choosing a class of admissible regions (for instance, bounded open sets satisfying suitable causal regularity conditions) and building them into a category. In the simplest setup this is a poset $\mathcal{R}(M)$, with morphisms given by inclusions. A (Haag-Kastler) net on $M$ can then be viewed as a covariant functor
$$
\mathcal{A}:\mathcal{R}(M)\to \text{Alg},
$$
where $\text{Alg}$ is a category of operator-algebraic objects (e.g. $*$-algebras, $C^*$-algebras, or von Neumann algebras) with structure-preserving morphisms. In practice one often restricts to injective (or monic) morphisms, so that for an inclusion $O\subseteq O'$ the map $\mathcal{A}(O)\to \mathcal{A}(O')$ identifies $\mathcal{A}(O)$ with a subalgebra of $\mathcal{A}(O')$ via its image. Conceptually, the net records (i) which observables are available in each region and (ii) how these collections relate when one enlarges the region.

The main physical constraints are imposed as conditions on the functor $\mathcal{A}$. These are the Haag-Kastler axioms: isotony, locality, and covariance, often supplemented by additivity, and on globally hyperbolic spacetimes a time-slice condition expressing propagation from Cauchy data \cite{HaagKastler1964}. Categorically, isotony is simply functoriality along inclusions, together with the requirement that the corresponding maps are injective when one wants to regard $\mathcal{A}(O)$ as a subalgebra of $\mathcal{A}(O')$. Locality requires that observables assigned to spacelike separated regions commute inside the algebra of any larger region containing them. Covariance expresses compatibility with the relevant spacetime symmetries (or, in more flexible setups, with a class of admissible embeddings). Additivity packages the idea that algebras of larger regions are generated by algebras of smaller regions that cover them, in an appropriate sense. The time-slice axiom singles out regions containing a Cauchy surface and requires that their observables generate the full algebra.

It is often helpful to keep the typing explicit by separating the region category $\mathcal{R}(M)$ from the target $\text{Alg}$. For example, isotony may be displayed by applying $\mathcal{A}$ to an inclusion chain in $\mathcal{R}(M)$:
\[
\begin{tikzcd}
	{O_{1}} & \subseteq & {O_{2}}  & \subseteq & {O_{3}} \\
	{\mathcal{A}(O_{1})} && {\mathcal{A}(O_{2})} && {\mathcal{A}(O_{3})}
	\arrow["{\mathcal{A}}", from=1-1, to=2-1]
	\arrow["{\mathcal{A}}", from=1-3, to=2-3]
	\arrow["{\mathcal{A}}", from=1-5, to=2-5]
	\arrow[from=2-1, to=2-3]
	\arrow[from=2-3, to=2-5]
\end{tikzcd}
\]
The constraint undetermined by functoriality alone is locality: if $O_{1}$ and $O_{2}$ are spacelike separated subregions of a larger region $O$, then the subalgebras $\mathcal{A}(O_{1})$ and $\mathcal{A}(O_{2})$ must commute inside $\mathcal{A}(O)$. In the fixed-background ``net on $M$'' formulation, this is typically imposed as an additional axiom describing how two subalgebras sit inside a common algebra, rather than something forced by functoriality alone.

Seen this way, the standard setup already hints at richer structure. Causal compatibility is not a morphism in the region poset; it is extra data about pairs of inclusions into a common target. Likewise, local-to-global principles (additivity, time-slice, and variants) are not consequences of the definition of a functor $\mathcal{A}$; they are further constraints governing how $\mathcal{A}$ behaves on certain families of inclusions (covers, unions, or distinguished embeddings).

\textbf{The locally covariant perspective.} The locally covariant framework replaces a single background spacetime by a category $\text{Loc}$ of spacetimes and admissible embeddings. In the formulation of Brunetti, Fredenhagen and Verch, the objects of $\text{Loc}$ are oriented and time-oriented globally hyperbolic Lorentzian manifolds, and the morphisms are isometric embeddings preserving the orientations and having causally convex open image \cite{Brunetti_2003}. Causal convexity ensures that an embedding can be read as inclusion of a subsystem: causal relations seen in the image are exactly those inherited from the domain.

A locally covariant quantum field theory is then a covariant functor
$$
\mathcal{A}:\text{Loc}\to \text{Alg},
$$
subject to additional physically motivated axioms stated as properties of $\mathcal{A}$ \cite{Brunetti_2003, fewster2015algebraicquantumfieldtheory}. The point is that covariance is built into the indexing geometry. Rather than fixing $M$ and adding symmetry actions by hand, one lets functoriality express how observables are transported along spacetime embeddings.

Functoriality is the ``kinematic'' part of the theory: it encodes consistent transport of observables under admissible embeddings, before imposing causality, time-slice, or other physical constraints. Concretely, an embedding $\psi:M\to N$ identifies $M$ with a causally well-behaved subspacetime of $N$, and functoriality supplies the induced homomorphism of observables
$$
\mathcal{A}(\psi):\mathcal{A}(M)\to \mathcal{A}(N).
$$
Accordingly, if $\varphi:L\to M$ and $\psi:M\to N$ are morphisms in $\text{Loc}$, then
$$
\mathcal{A}(\psi \circ \varphi)=\mathcal{A}(\psi)\circ \mathcal{A}(\varphi),
\qquad
\mathcal{A}(\mathrm{id}_{M})=\mathrm{id}_{\mathcal{A}(M)}.
$$
To keep typing explicit, one can display composition in $\text{Loc}$ and its image in $\text{Alg}$ side by side:
\[
\begin{tikzcd}
	L & M && {\mathcal{A}(L)} & {\mathcal{A}(M)} \\
	& N &&& {\mathcal{A}(N)}
	\arrow["\varphi", from=1-1, to=1-2]
	\arrow["{\psi \circ \varphi}"', from=1-1, to=2-2]
	\arrow["\psi", from=1-2, to=2-2]
	\arrow["{\mathcal{A}(\varphi)}", from=1-4, to=1-5]
	\arrow["{\mathcal{A}(\psi\circ \varphi)}"', from=1-4, to=2-5]
	\arrow["{\mathcal{A}(\psi)}", from=1-5, to=2-5]
\end{tikzcd}
\]

How does one recover a net on a fixed spacetime $M$ from $\mathcal{A}:\text{Loc}\to\text{Alg}$? Fix $M\in \text{Loc}$ and choose a class of subregions $O\subseteq M$ (for instance, open, relatively compact, causally convex regions) for which the restriction $M|_{O}$ is globally hyperbolic so ($M|_{O}\in \text{Loc}$) together with the induced inclusion $\iota_{O}:M|_{O}\to M$. Restricting along these inclusions yields a net
$$
O \longmapsto \mathcal{A}(M|_{O}),
\qquad
O\subseteq O' \longmapsto \mathcal{A}(\iota_{O'O}):\mathcal{A}(M|_{O})\to \mathcal{A}(M|_{O'}),
$$
where $\iota_{O'O}:M|_{O}\to M|_{O'}$ is the morphism in $\text{Loc}$ induced by the inclusion $O\subseteq O'$. In particular, the structure of isotony is provided by functoriality; when one works in a target category whose morphisms are injective, isotony is built into the basic typing.

Locality and time-slice enter as additional requirements. Causality constrains pairs of morphisms with common codomain: if $\varphi_{1}:M_{1}\to N$ and $\varphi_{2}:M_{2}\to N$ have causally disjoint images, then the subalgebras $\mathcal{A}(\varphi_{1})(\mathcal{A}(M_{1}))$ and $\mathcal{A}(\varphi_{2})(\mathcal{A}(M_{2}))$ commute inside $\mathcal{A}(N)$. Time-slice is imposed by selecting a class $S$ of Cauchy morphisms and requiring $\mathcal{A}(s)$ to be an isomorphism for every $s\in S$ \cite{Brunetti_2003, fewster2015algebraicquantumfieldtheory}. In this way, the locally covariant framework cleanly separates the covariant functorial assignment from the genuinely physical constraints.

A further refinement is to treat fields and comparisons between theories via natural transformations. In the locally covariant setting, a map between two theories is a natural transformation between the corresponding functors $\text{Loc}\to\text{Alg}$, as emphasized in Fewster and Verch's account of the category of locally covariant theories \cite{fewster2015algebraicquantumfieldtheory}. This also makes clear what the basic $1$-categorical setup does not capture: in many situations the right notion of ``sameness'' of theories is weaker than equality and is better expressed in terms of equivalences together with coherent comparison data, motivating higher-categorical and homotopical refinements.

\textbf{Orthogonal categories.} A recurring theme in categorical AQFT is that the key constraints are not about single embeddings in isolation, but about pairs and families of embeddings that are compatible in a causal sense. The notion of an orthogonal category packages this neatly.

An orthogonal category $(C,\perp)$ is a category $C$ together with an orthogonality relation on pairs of morphisms with a common codomain, required to be stable under pre- and post-composition. The basic example takes $C=\text{Loc}$ and declares two morphisms orthogonal when their images are causally disjoint. There is also a fixed-spacetime version: for $M\in\text{Loc}$, the slice category $\text{Loc}/M$ inherits an induced orthogonality relation and can be used to model theories akin to Haag--Kastler nets on the single spacetime $M$ \cite{Benini_2021}.

In this approach, locality is built into the indexing data. One works with $(C,\perp)$ and a functor $\mathcal{A}:C\to \text{Alg}_{k}$, where $\text{Alg}_{k}$ is a chosen symmetric monoidal category of $k-$linear observables. The relation $f_{1}\perp f_{2}$ is defined for pairs of morphisms $f_{1}:c_{1}\to t$ and $f_{2}:c_{2}\to t$ with common codomain $t$. Locality is then the requirement that whenever $f_{1}\perp f_{2}$, the subalgebras $\mathcal{A}(f_{1})(\mathcal{A}(c_{1}))$ and $\mathcal{A}(f_{2})(\mathcal{A}(c_{2}))$ commute inside $\mathcal{A}(t)$. Equivalently, for all $a_{1}\in \mathcal{A}(c_{1})$ and $a_{2}\in \mathcal{A}(c_{2})$,
$$
[\mathcal{A}(f_{1})(a_{1}),\, \mathcal{A}(f_{2})(a_{2})]=0
\hspace{0.2cm}\text{in}\hspace{0.1cm}\mathcal{A}(t).
$$
This is the usual locality axiom, now stated so that the causal input appears as part of the domain structure through $\perp$.

This viewpoint interfaces directly with operads, which encode multi-input composition. From $(C,\perp)$ one constructs a colored operad $P_{C}$ whose operations are finite tuples of pairwise orthogonal morphisms $(f_{1},\ldots,f_{n})$ with a common target. Algebras over $P_{C}$ give the standard operadic presentation of theories equipped with the corresponding multi-region structure maps, compatible with orthogonality. In the basic geometric example $C=\text{Open}(M)$ with $U\perp V$ when $U\cap V=\varnothing$, one recovers the prefactorization operad from factorization algebra theory \cite{Benini_2021}.

In this formulation, an AQFT can be presented as an algebra over $P_{C}$ in a chosen symmetric monoidal target. In that presentation, the $\perp$-commutativity requirement is built into the operadic relations rather than imposed separately. This operadic reformulation is especially useful in homotopical and derived treatments of AQFT, where model structures on categories of operadic algebras provide a systematic language for weak equivalences of theories and for derived constructions \cite{carmona2023newmodelstructuresalgebraic}. The payoff is that locality is no longer an extra axiom attached to a $1$-functor; it becomes part of the algebraic data of the theory and behaves well under standard categorical operations, including time-slice localizations, extension along embeddings of orthogonal categories, and homotopy-invariant replacements \cite{carmona2023newmodelstructuresalgebraic}\cite{Benini_2021}.

\textbf{Kan extensions.} Two different Kan extension constructions show up naturally in categorical AQFT, because there are two different ways one wants to pass from local data to global data.

First, in the orthogonal--operadic setting, one often has a full embedding of orthogonal categories
$$
J: C \hookrightarrow D.
$$
Given a theory $\mathcal{A}\in \text{AQFT}(C)$, there is a canonical way to extend it to $D$: one takes the operadic left Kan extension along the induced morphism of colored operads $P_{C}\to P_{D}$. This produces a theory $J_{!}\mathcal{A}\in \text{AQFT}(D)$ and fits into an adjunction
$$
J_{!}\dashv J^{*},
$$
where $J^{*}$ is restriction along $J$. Roughly, $J_{!}\mathcal{A}$ is the universal extension of $\mathcal{A}$ to the larger indexing category $D$, obtained by freely adding precisely the composites forced by the operadic locality structure and the orthogonality relations. In familiar geometric cases, this universal extension recovers what is often called Fredenhagen's universal algebra \cite{Benini_2021}.

Second, in gauge-theoretic and structured spacetime settings, one starts from a category $\text{Str}$ fibered in groupoids over $\text{Loc}$, with projection
$$
\pi: \text{Str}\to \text{Loc}
$$
forgetting extra background structure (such as spin structures or principal bundles with connection). A theory on structured spacetimes is a functor $\mathcal{A}:\text{Str}\to \text{Alg}$. One then wants a theory on $\text{Loc}$ that packages, in a controlled way, how $\mathcal{A}$ depends on the additional structure. A canonical candidate is the right Kan extension
$$
\text{Ran}_{\pi}\mathcal{A}:\text{Loc}\to \text{Alg},
$$
equipped with its universal natural transformation
$$
\epsilon: (\text{Ran}_{\pi}\mathcal{A})\circ \pi \Rightarrow \mathcal{A}.
$$
Diagrammatically:
\[
\begin{tikzcd}
	{\text{Str}} & {\text{Alg}} \\
	{\text{Loc}} & {\text{Alg}}
	\arrow["{\mathcal{A}}", from=1-1, to=1-2]
	\arrow["\pi"', from=1-1, to=2-1]
	\arrow["{\text{Ran}_{\pi}\mathcal{A}}"', from=2-1, to=2-2]
\end{tikzcd}
\quad
\epsilon: (\text{Ran}_{\pi}\mathcal{A})\circ \pi \Rightarrow \mathcal{A}
\]
The universal property is the usual one: $(\text{Ran}_{\pi}\mathcal{A})\circ \pi \Rightarrow \mathcal{A}$ is terminal among pairs $\mathcal{B}:\text{Loc}\to \text{Alg}$ equipped with a natural transformation $\mathcal{B}\circ \pi \Rightarrow \mathcal{A}$. In concrete models, $\text{Ran}_{\pi}\mathcal{A}(M)$ can often be computed as a limit over the groupoid of $\pi$-structures on $M$, which is the mechanism that packages the dependence on background structure into a single $\text{Loc}$-theory.

This pointwise description also explains why strict isotony can fail after forgetting structure. A morphism $f:M\to N$ in $\text{Loc}$ induces a comparison map $\text{Ran}_{\pi}\mathcal{A}(f):\text{Ran}_{\pi}\mathcal{A}(M)\to \text{Ran}_{\pi}\mathcal{A}(N)$ between two limits, taken over the structure groupoids above $M$ and $N$. Without additional extension hypotheses on $\pi$, restriction of structures along $f$ need not be well-behaved enough to make this comparison map injective. Concretely, elements that are nontrivial at $M$ can be forced to vanish when transported to $N$. This is the logic behind the isotony violations exhibited in structured and gauge-theoretic settings \cite{Benini_2017}.

Together, these Kan extension constructions formalize two complementary local-to-global procedures: extension along inclusions of indexing geometries (operadic left Kan extension), and descent from structure-dependent observables to a theory on underlying spacetimes (right Kan extension).

\textbf{2AQFT.} The operadic reformulation already suggests a natural categorification. In the algebra-valued setting, a region is assigned an algebra, and compatible families of regions are combined using multiplication maps. In many contexts, especially with symmetry and descent, the structure one wants to assign to a region is not a single algebra but a $k$-linear category of objects, with restriction and combination implemented by functors. This leads to $2$-categorical algebraic quantum field theory.

Fix an orthogonal category $(C,\perp)$ and its associated prefactorization operad $P_{C}$. A $2$-categorical AQFT on $C$ keeps the same operad $P_{C}$ but changes the target from algebras to the $2$-category $\text{Pr}_{k}$ of locally presentable $k$-linear categories. Concretely, one sets
$$
\text{2AQFT}(C):=\text{Alg}_{P_{C}}(\text{Pr}_{k}).
$$
Thus a theory $\mathcal{A}\in \text{2AQFT}(C)$ assigns to each object $c\in C$ a locally presentable $k$-linear category $\mathcal{A}(c)$. For each finite tuple of pairwise orthogonal morphisms with common codomain
$$
f_{1}:c_{1}\to t,\ldots,f_{n}:c_{n}\to t,
$$
it assigns a factorization product functor
$$
\mathcal{A}(f_1,\dots,f_n)\;:\;\prod_{i=1}^n \mathcal{A}(c_i)\longrightarrow \mathcal{A}(t),
$$
which is $k$-linear and cocontinuous in each variable. The nullary operation supplies a distinguished object of $\mathcal{A}(t)$, playing the role of a unit. In this way, the same multi-region locality encoded by $P_C$ acts by multilinear functors between categories, rather than by multiplications in algebras \cite{Benini_2021}.

Two structural points are worth emphasizing. First, $\mathrm{Pr}_k$ is symmetric monoidal via a tensor product, so multi-input operations are implemented directly in the target. Second, higher comparison data is intrinsic: morphisms of theories are functors between the assigned $k$-linear categories, and $2$-morphisms are natural transformations. This is well suited to situations where one expects identifications only up to equivalence, together with coherent comparison data, as is typical in the presence of gauge symmetry or descent \cite{Benini_2021}.

\textbf{Homotopy.} The orthogonal--operadic presentation of AQFT admits a natural homotopical refinement. If one chooses a symmetric monoidal model category of observables (such as chain complexes), then an AQFT can be expressed as an algebra over $P_C$ in that model category. One can then put a model structure on the resulting category of theories, often with weak equivalences detected objectwise (or by criteria tailored to the application). The guiding principle is that derived constructions (derived quantization, derived extensions, derived local-to-global procedures) should depend only on the weak-equivalence class of the theory, not on a particular strict presentation \cite{carmona2023newmodelstructuresalgebraic}.

A key point is that the time-slice condition can be expressed as a localization of the indexing geometry. Fix a distinguished class $S$ of morphisms in an orthogonal category $C$, for instance the Cauchy morphisms. Form the localized category $C[S^{-1}]$. One can then compare time-slice theories with operadic algebras over the corresponding localized operad. In this formulation, time-slice is not an extra axiom attached to a functor; it is the statement that the theory factors through a category in which the morphisms in $S$ have been inverted. This fits well with homotopical methods: localizations and Quillen equivalences give principled ways to compare different models presenting the same physical content \cite{carmona2023newmodelstructuresalgebraic}.

Even when one ultimately works with operator algebras rather than chain complexes, the conceptual message remains. Orthogonal categories and their operads isolate the multi-region causal input and make it available to standard categorical constructions. The homotopical viewpoint supplies a notion of equivalence that is stable under those constructions. Together they explain why a strictly $1$-categorical formulation can miss structure: many comparisons in field theory are naturally formulated only up to equivalence with coherent compatibility data, and many local-to-global procedures are governed by universal properties (Kan extensions, localizations) that become more transparent once higher-categorical features are allowed to appear.

\textbf{Toward Double Categories.} The viewpoints above can be read as a steady shift in what we ask the formalism to carry: instead of treating the main physical requirements as side conditions on a bare assignment of observables, we keep enriching the categorical setup until those requirements begin to look like structure.

In the fixed-spacetime net picture and in locally covariant AQFT, the basic move is still a $1$-functor assigning observables to regions or spacetimes. Locality and time-slice then enter as extra conditions on that functor, formulated using a causal disjointness relation on pairs of embeddings into a common target and a chosen class of Cauchy morphisms required to map to isomorphisms \cite{Brunetti_2003, fewster2015algebraicquantumfieldtheory}. Orthogonal categories and their operadic encoding sharpen this point: the causal input is recorded directly in the domain via $\perp$, and locality is enforced by the operadic structure, so that observables from pairwise $\perp$-compatible regions commute when they are combined inside a common target \cite{carmona2023newmodelstructuresalgebraic, Benini_2021}.

Kan extensions add a second guiding principle. They provide canonical ways to extend a theory along inclusions and to descend from structured to unstructured spacetimes, and they make clear why strict properties such as isotony can break down once invariance and descent are built in \cite{Benini_2017}.

Finally, $2$AQFT upgrades the target from algebras to linear categories, so that comparisons between theories and the coherence data they carry are part of the basic language, and universal extension procedures admit categorified refinements \cite{Benini_2021}.

What still remains, even after these upgrades, is a clean way to keep two different kinds of structure separate while letting them interact. AQFT uses inclusions of regions to express restriction and extension of observables, and these are naturally encoded as morphisms in a region category or in $\text{Loc}$. At the same time, AQFT repeatedly invokes relations that are not inclusions, such as causal independence and more general correlation data. In the $1$-categorical setting these typically appear as commutation requirements inside a larger algebra, and in operator-algebraic practice they are often mediated by bimodules or couplings that relate algebras without identifying them. Orthogonality captures an important part of this second story, but only as a relation on pairs of arrows with a common codomain. Likewise, $2$AQFT enriches the target, but its domain remains a $1$-category equipped with $\perp$, so non-inclusion relations still enter indirectly through commutativity constraints, extension procedures, and descent.

This motivates a framework in which inclusions and correlation-type relations form two compositional directions, together with explicit coherence data expressing their compatibility. Double categories provide exactly this separation. One direction records inclusions of regions, so isotony becomes functorial in that direction. The other direction records the additional relations one wants to track, and squares express that restriction along inclusions is compatible with transport along these non-inclusion morphisms.

In the next section we implement this idea concretely. We construct double categories $\mathbf{Mink}(M)$ and $\mathbf{vNA}$ and define a double functor
$$
\mathbb{F}:\mathbf{Mink}(M)\to \mathbf{vNA},
$$
sending a region $U$ to its algebra $\mathcal{A}(U)$, inclusions to the usual $*$-homomorphisms, and horizontal causal data to Hilbert bimodules composed by Connes fusion. The resulting squares package the compatibility conditions that, in the $1$-categorical formulation, are imposed externally as axioms.

\section{Double Categorical AQFT}\label{sec:double}

\textbf{Notation and Tools.} Throughout this section, every double category is a \emph{pseudo} double category: an internal category object in $\mathbf{Cat}$ where the horizontal direction is weak (associative and unital only up to coherent isomorphism), while vertical composition is strict. Intuitively: vertical arrows behave like honest morphisms in an ordinary category, and horizontal arrows compose like $1$-cells in a bicategory.

\begin{definition}[Pseudo double category]
A pseudo double category $\mathbb{D}$ consists of:
\begin{enumerate}[label=\alph*)]
  \item a category $\mathbb{D}_0$ (objects and vertical morphisms),
  \item a category $\mathbb{D}_1$ (objects are horizontal morphisms, morphisms are squares),
  \item functors
  $$
  s,t:\mathbb{D}_1 \rightrightarrows \mathbb{D}_0
  $$
  (horizontal source/target),
  \item a functor
  $$
  U:\mathbb{D}_0 \to \mathbb{D}_1
  $$
  (horizontal identity assignment),
  \item and a horizontal composition functor
  $$
  \odot:\mathbb{D}_1 \times_{\mathbb{D}_0} \mathbb{D}_1 \longrightarrow \mathbb{D}_1,
  $$
  defined on the pullback of $t$ and $s$ (so $(H,K)$ is composable exactly when $t(H)=s(K)$).
\end{enumerate}
These data make
$$
(\mathbb{D}_0,\mathbb{D}_1,s,t,U,\odot)
$$
an internal category in $\mathbf{Cat}$ up to specified coherent natural isomorphisms (a horizontal associator and left/right unitors satisfying the pentagon and triangle identities). Vertical composition is strictly the categorical composition in $\mathbb{D}_0$ and $\mathbb{D}_1$.
\end{definition}

An object of $\mathbb{D}_1$ is a horizontal morphism. If $H \in \mathrm{Ob}(\mathbb{D}_1)$, we write
$$
H: A \Rightarrow B
$$
to mean
$$
s(H)=A, \quad t(H)=B,
$$
where $s$ and $t$ denote the horizontal source and target in $\mathbb{D}_0$.

A morphism in $\mathbb{D}_1$ is a square. Concretely, a square is a morphism
$$
\alpha: H \longrightarrow K
$$
in the category $\mathbb{D}_1$, where $H,K \in \mathrm{Ob}(\mathbb{D}_1)$ are horizontal morphisms. Thus, if
$$
H: A \Rightarrow B,
\quad
K: A' \Rightarrow B',
$$
then applying the functors $s,t:\mathbb{D}_1 \to \mathbb{D}_0$ to the morphism $\alpha:H\to K$ produces the vertical boundary maps
$$
s(\alpha): A=s(H) \longrightarrow s(K)=A',
\qquad
t(\alpha): B=t(H) \longrightarrow t(K)=B'
$$
in $\mathbb{D}_0$. We often denote these boundary maps by
$$
f := s(\alpha): A \to A',
\qquad
g := t(\alpha): B \to B'.
$$

% https://q.uiver.app/#q=WzAsNCxbMCwwLCJBIl0sWzAsMSwiQSciXSxbMSwwLCJCIl0sWzEsMSwiQiciXSxbMCwxLCJmIiwyXSxbMiwzLCJnIl0sWzAsMiwiSCJdLFsxLDMsIksiLDJdXQ==
\[\begin{tikzcd}
	A & B \\
	{A'} & {B'}
	\arrow["H", from=1-1, to=1-2]
	\arrow["f"', from=1-1, to=2-1]
	\arrow["g", from=1-2, to=2-2]
	\arrow["K"', from=2-1, to=2-2]
\end{tikzcd}\]

In the depiction above, $H$ is the top horizontal arrow, $K$ is the bottom horizontal arrow, and $(f,g)$ are the left and right vertical arrows determined by the square $\alpha$.

Recall that the functor
$$
U:\mathbb{D}_0 \to \mathbb{D}_1
$$
assigns horizontal identities and their unit squares. It has two distinct typings, and it is worth keeping both straight:

\begin{enumerate}[label=\alph*)]
\item For each object $A \in \mathrm{Ob}(\mathbb{D}_0)$, the object $U(A) \in \mathrm{Ob}(\mathbb{D}_1)$ is the horizontal identity at $A$. We write
$$
U(A) =: U_A : A \Rightarrow A,
$$
so that
$$
s(U_A)=A=t(U_A).
$$

\item For each vertical morphism $f:A\to A'$ in $\mathbb{D}_0$, the morphism $U(f):U_A \to U_{A'}$ in $\mathbb{D}_1$ is a square whose vertical boundary is $(f,f)$, i.e.
$$
s\bigl(U(f)\bigr)=f,
\qquad
t\bigl(U(f)\bigr)=f.
$$
We denote this unit square by
$$
\iota_f := U(f).
$$
Equivalently, $\iota_f$ may be depicted as:

% https://q.uiver.app/#q=WzAsNCxbMCwwLCJBIl0sWzEsMCwiQSJdLFswLDEsIkEnIl0sWzEsMSwiQSciXSxbMCwxLCJVX3tBfSJdLFsyLDMsIlVfe0EnfSJdLFswLDIsImYiLDJdLFsxLDMsImYiXV0=
\[\begin{tikzcd}
	A & A \\
	{A'} & {A'}
	\arrow["{U_{A}}", from=1-1, to=1-2]
	\arrow["f"', from=1-1, to=2-1]
	\arrow["f", from=1-2, to=2-2]
	\arrow["{U_{A'}}", from=2-1, to=2-2]
\end{tikzcd}\]

\end{enumerate}

For any vertical morphism $\varphi$ in $\mathbb{D}_0$, $U(\varphi)$ is a square whose vertical boundary is $(\varphi,\varphi)$ and whose horizontal source and target are horizontal identities. In particular, $U(\varphi)$ cannot have prescribed vertical boundary $(\varphi_1,\varphi_2)$ unless $\varphi_1=\varphi_2=\varphi$.

Throughout the construction, we make use of Juan Orendain's free globularly generated double category construction, detailed in \cite{orendain2019freeglobularilygenerateddouble}\cite{orendain2021freeglobularlygenerateddouble}. To this end, we recall the pieces of Orendain's setup that we will actually use later.

\begin{definition}[Globular square]
A square
$$
\alpha: H \longrightarrow K
$$
is globular if its vertical boundary maps are identities; equivalently, if
$$
s(H)=s(K),\quad t(H)=t(K),\quad
s(\alpha)=\mathrm{id}_{s(H)},\quad
t(\alpha)=\mathrm{id}_{t(H)}.
$$
Equivalently, it is a square of the form:

% https://q.uiver.app/#q=WzAsNCxbMCwwLCJBIl0sWzEsMCwiQiJdLFswLDEsIkEiXSxbMSwxLCJCIl0sWzAsMiwiXFx0ZXh0e2lkfSIsMl0sWzEsMywiXFx0ZXh0e2lkfSJdLFswLDEsIkgiXSxbMiwzLCJLIiwyXV0=
\[\begin{tikzcd}
	A & B \\
	A & B
	\arrow["H", from=1-1, to=1-2]
	\arrow["{\text{id}}"', from=1-1, to=2-1]
	\arrow["{\text{id}}", from=1-2, to=2-2]
	\arrow["K"', from=2-1, to=2-2]
\end{tikzcd}\]
\end{definition}

Fix a pseudo double category $\mathbb{C}$. Write $G(\mathbb{C})$ for its collection of globular squares. For each vertical morphism $f \in \mathrm{Mor}(\mathbb{C}_0)$, write
$
\iota_f := U(f)
$
for the corresponding unit square.

\begin{definition}[$0$-marked squares]
Let $\mathbb{C}$ be a pseudo double category. The set of $0$-marked squares in $\mathbb{C}$ is
$$
G_{\mathbb{C}}
:=
G(\mathbb{C}) \,\cup\, \{\iota_f \mid f \in \mathrm{Mor}(\mathbb{C}_0)\},
$$
where $G(\mathbb{C})$ denotes the collection of globular squares in $\mathbb{C}$ and $\iota_f:=U(f)$ is the unit square associated to a vertical morphism $f$.
\end{definition}

\begin{definition}[Globularly generated double category] 
Following Orendain, a pseudo double category $\mathbb{C}$ is globularly generated if every square of $\mathbb{C}$ lies in the smallest class of squares that contains $G_{\mathbb{C}}$ and is closed under vertical and horizontal composition. Equivalently, the closure of $G_{\mathbb{C}}$ under the double-category operations is all of $\mathrm{Mor}(\mathbb{C}_1)$.
\end{definition}

\begin{definition}[$\gamma$-construction]
Define $\gamma\mathbb{C}$ to be the smallest sub-double-category of $\mathbb{C}$ which:
\begin{enumerate}[label=\alph*)]
\item has the same objects and vertical morphisms as $\mathbb{C}$,
\item has the same horizontal morphisms as $\mathbb{C}$,
\item and whose squares contain $G_{\mathbb{C}}$ and are closed under vertical and horizontal composition (hence under whiskering).
\end{enumerate}
By construction, $\gamma\mathbb{C}$ is globularly generated.
\end{definition}

\begin{definition}[Decorated horizontalization $H^\ast$]
Let $\mathbb{C}$ be a pseudo double category. The decorated horizontalization $H^\ast(\mathbb{C})$ is the decorated bicategory with:
\begin{enumerate}[label=\alph*)]
\item objects $\mathrm{Ob}(\mathbb{C}_0)$,
\item decoration given by the vertical category $\mathbb{C}_0$,
\item horizontal $1$-cells $A\to B$ given by horizontal morphisms $H\in \mathrm{Ob}(\mathbb{C}_1)$ with $s(H)=A$ and $t(H)=B$,
\item $2$-cells given by globular squares $\alpha:H\Rightarrow K$, i.e. morphisms $\alpha:H\to K$ in $\mathbb{C}_1$ with identity vertical boundary,
\item horizontal composition and units induced by $\odot$ and $U$, with associator and unitors inherited from the pseudo double category structure.
\end{enumerate}
Thus $H^\ast(\mathbb{C})$ retains horizontal composition together with globular $2$-cells, and retains $\mathbb{C}_0$ as decoration, but discards non-globular square data.
\end{definition}

\begin{theorem}[Minimality of $\gamma\mathbb{C}$ at fixed $H^\ast$]
It is proven that
$$
H^\ast(\gamma\mathbb{C}) = H^\ast(\mathbb{C}),
$$
and that $\gamma\mathbb{C}$ is contained in every sub-double-category $\mathbb{D}\subseteq \mathbb{C}$ satisfying
$$
H^\ast(\mathbb{D}) = H^\ast(\mathbb{C}).
$$
In particular, $\gamma\mathbb{C}$ is the minimal sub-double-category of $\mathbb{C}$ with the same decorated horizontalization \cite{orendain2021freeglobularlygenerateddouble}.
\end{theorem}

Let $\mathbf{bCat}^\ast$ denote Orendain's category of decorated bicategories and decorated pseudofunctors, and let $\mathbf{gCat}$ denote the category of globularly generated pseudo double categories and double functors.

\begin{theorem}[Adjunction $Q \dashv H^\ast$ on globularly generated double categories]
We can construct a left adjoint
$$
Q:\mathbf{bCat}^\ast \;\rightleftarrows\; \mathbf{gCat}: H^\ast\!\!\restriction_{\mathbf{gCat}},
$$
with unit
$$
j_{\mathbb{B}}:\mathbb{B}\longrightarrow H^\ast Q(\mathbb{B}),
$$
and counit
$$
\pi_{\mathbb{C}}:QH^\ast(\mathbb{C})\longrightarrow \mathbb{C}
$$
for globularly generated $\mathbb{C}$ \cite{orendain2021freeglobularlygenerateddouble}.
\end{theorem}

The direction of the unit is substantive: $Q$ is free, so applying $H^\ast$ after free generation can enlarge the decorated bicategory one started with.

\textit{Remark.} The unit $j_{\mathbb{B}}:\mathbb{B}\to H^\ast Q(\mathbb{B})$ identifies $\mathbb{B}$ with its image inside $H^\ast Q(\mathbb{B})$. In general there is no canonical map $H^\ast Q(\mathbb{B})\to \mathbb{B}$, reflecting that free generation may introduce new composites and hence additional globular squares in $H^\ast Q(\mathbb{B})$ \cite{orendain2019freeglobularilygenerateddouble}\cite{orendain2021freeglobularlygenerateddouble}.

\begin{definition}[Saturation]
The saturation of a decorated bicategory $\mathbb{B}$ is $H^\ast Q(\mathbb{B})$. One calls $\mathbb{B}$ saturated if
$$
\mathbb{B}=H^\ast Q(\mathbb{B}),
$$
equivalently if the unit $j_{\mathbb{B}}:\mathbb{B}\to H^\ast Q(\mathbb{B})$ is an isomorphism \cite{orendain2019freeglobularilygenerateddouble}.
\end{definition}

\textit{Remark.} Saturation is not automatic. Without saturation, the free composites in $Q(\mathbb{B})$ can induce globular $2$-cells in $H^\ast Q(\mathbb{B})$ not present in $\mathbb{B}$, so $Q(\mathbb{B})$ need not internalize $\mathbb{B}$. By contrast, $Q\bigl(H^\ast Q(\mathbb{B})\bigr)$ always internalizes the saturation \cite{orendain2019freeglobularilygenerateddouble}.

\textbf{Double Category of Spacetime.} We now define the domain of our AQFT double functor: a double category of regions in a fixed spacetime $M$, denoted $\mathbf{Mink}(M)$. The philosophy is simple: vertical arrows are inclusions (the isotony direction), horizontal arrows are admissible embeddings (the covariance/transport direction), and squares are just commuting squares of embeddings.

Fix a spacetime $M$. Let $\mathcal{R}(M)$ be a chosen class (or set) of admissible regions.

\begin{definition}[Embedding category]
Let $\mathbf{Emb}(M)$ be the category whose objects are regions $U\in \mathcal{R}(M)$ and whose morphisms $h:U\hookrightarrow V$ are admissible embeddings. Identities are admissible and admissible embeddings are closed under composition, hence $\mathbf{Emb}(M)$ is well-defined.
\end{definition}

Write $\mathbf{Inc}(M)\subseteq \mathbf{Emb}(M)$ for the wide subcategory of inclusions and $\mathbf{Iso}(M)\subseteq \mathbf{Emb}(M)$ for the subgroupoid of admissible isomorphisms.

A basic piece of geometry is a commuting square in $\mathbf{Emb}(M)$,
% https://q.uiver.app/#q=WzAsNCxbMSwwLCJWIl0sWzAsMSwiVSciXSxbMSwxLCJWJyJdLFswLDAsIlUiXSxbMCwyLCJqIl0sWzEsMiwiaCciLDJdLFszLDEsImkiLDJdLFszLDAsImgiXV0=
\[\begin{tikzcd}
	U & V \\
	{U'} & {V'}
	\arrow["h", from=1-1, to=1-2]
	\arrow["i"', from=1-1, to=2-1]
	\arrow["j", from=1-2, to=2-2]
	\arrow["{h'}"', from=2-1, to=2-2]
\end{tikzcd}\quad \text{meaning}\quad j \circ h = h'\circ i. \] 

\begin{definition}[Spacetime double category $\mathbf{Mink}(M)$]
Define a (strict) double category $\mathbf{Mink}(M)$ as follows:
\begin{enumerate}[label=\alph*)]
\item $\mathbf{Mink}(M)_0 := \mathbf{Inc}(M)$. Thus the objects are regions and the vertical morphisms are inclusions.
\item The horizontal morphisms $U\Rightarrow V$ are the admissible embeddings $h:U\hookrightarrow V$, i.e. the morphisms of $\mathbf{Emb}(M)$.
\item A square in $\mathbf{Mink}(M)$ from $h:U\hookrightarrow V$ to $h':U'\hookrightarrow V'$ with vertical boundary $i:U\hookrightarrow U'$ and $j:V\hookrightarrow V'$ is a commuting square in $\mathbf{Emb}(M)$:
% https://q.uiver.app/#q=WzAsNCxbMSwwLCJWIl0sWzAsMSwiVSciXSxbMSwxLCJWJyJdLFswLDAsIlUiXSxbMCwyLCJqIl0sWzEsMiwiaCciLDJdLFszLDEsImkiLDJdLFszLDAsImgiXV0=
\[\begin{tikzcd}
	U & V \\
	{U'} & {V'}
	\arrow["h", from=1-1, to=1-2]
	\arrow["i"', from=1-1, to=2-1]
	\arrow["j", from=1-2, to=2-2]
	\arrow["{h'}"', from=2-1, to=2-2]
\end{tikzcd} \quad (j \circ h = h' \circ i) \]
with $i,j$ inclusions.
\item Horizontal composition is composition in $\mathbf{Emb}(M)$; vertical composition is composition in $\mathbf{Inc}(M)$; squares compose by pasting commuting squares.
\item Horizontal identities are the identity embeddings $\mathrm{id}_U:U\hookrightarrow U$. The unit square on an inclusion $i:U\hookrightarrow U'$ is the commuting square between identity embeddings:
% https://q.uiver.app/#q=WzAsNCxbMCwxLCJVJyJdLFsxLDEsIlUnIl0sWzAsMCwiVSJdLFsxLDAsIlUiXSxbMCwxLCJcXHRleHR7aWR9IiwyXSxbMiwwLCJpIiwyXSxbMiwzLCJcXHRleHR7aWR9Il0sWzMsMSwiaSJdXQ==
\[\begin{tikzcd}
	U & U \\
	{U'} & {U'}
	\arrow["{\text{id}}", from=1-1, to=1-2]
	\arrow["i"', from=1-1, to=2-1]
	\arrow["i", from=1-2, to=2-2]
	\arrow["{\text{id}}"', from=2-1, to=2-2]
\end{tikzcd}\]
\end{enumerate}
\end{definition}

\begin{lemma}
    $\mathbf{Mink}(M)$ is a strict double category.
\end{lemma}

\begin{proof}
Horizontal composition is strictly associative and unital because it is composition in the $1$-category $\mathbf{Emb}(M)$. Vertical composition is strictly associative and unital because it is composition in the $1$-category $\mathbf{Inc}(M)$. Squares compose by pasting commuting squares in $\mathbf{Emb}(M)$, and the interchange law holds strictly since both iterated composites express the same outer-rectangle commutativity in a $1$-category.
\end{proof}

\textit{Remark.} A globular square in $\mathbf{Mink}(M)$ has identity inclusions on its vertical sides:
% https://q.uiver.app/#q=WzAsNCxbMCwxLCJVIl0sWzEsMSwiViJdLFswLDAsIlUiXSxbMSwwLCJWIl0sWzAsMSwiaCciLDJdLFsyLDAsIlxcdGV4dHtpZH0iLDJdLFsyLDMsImgiXSxbMywxLCJcXHRleHR7aWR9Il1d
\[\begin{tikzcd}
	U & V \\
	U & V
	\arrow["h", from=1-1, to=1-2]
	\arrow["{\text{id}}"', from=1-1, to=2-1]
	\arrow["{\text{id}}", from=1-2, to=2-2]
	\arrow["{h'}"', from=2-1, to=2-2]
\end{tikzcd}\]
which forces $h=h'$. Thus $G(\mathbf{Mink}(M))$ consists essentially of identity squares on horizontal morphisms. In particular, $\gamma(\mathbf{Mink}(M))$ can be much smaller than $\mathbf{Mink}(M)$; it is useful for minimality/uniqueness bookkeeping, but not as a replacement for the spacetime geometry itself.

\textbf{Double Category of Observables.} We now define the codomain of our AQFT double functor: a pseudo double category of von Neumann algebras, denoted $\mathbf{vNA}$. Vertically, we keep the usual net direction (normal unital $*$-homomorphisms). Horizontally, we use correspondences so that composition is Connes fusion. Squares are then bimodular intertwiners, and commuting spacetime squares become well-typed squares of correspondences.

\begin{definition}[Vertical morphism class $\mathbf{vN}_{\mathrm{vert}}$]
Fix a class $V$ of von Neumann algebras and a subcategory $\mathbf{vN}_{\mathrm{vert}}$ with $\mathrm{Ob}(\mathbf{vN}_{\mathrm{vert}})=V$, whose morphisms are those normal unital $*$-homomorphisms $\varphi:A\to B$ for which the operator-algebraic framework used below provides a functorial standard-form assignment
$$
L^2(\varphi):L^2(A)\longrightarrow L^2(B)
$$
compatible with the correspondence structure and Connes fusion \cite{orendain2021freeglobularlygenerateddouble}.
\end{definition}

\begin{definition}[Correspondence]
For von Neumann algebras $A,B$, an $A$-$B$ correspondence
$$
H: A \rightsquigarrow B
$$
is a Hilbert space $H$ equipped with commuting normal unital $*$-representations
$$
\lambda_H: A \to B(H),
\quad
\rho_H: B^{\mathrm{op}} \to B(H).
$$
Equivalently, $H$ carries a normal left $A$-action and a normal right $B$-action that commute, meaning
$$
\lambda_H(a)\,\rho_H(b^{\mathrm{op}})=\rho_H(b^{\mathrm{op}})\,\lambda_H(a)
\qquad
(a\in A,\ b\in B).
$$
Equivalently, $\lambda_H(A)$ and $\rho_H(B^{\mathrm{op}})$ are commuting von Neumann subalgebras of $B(H)$.
\end{definition}

\begin{definition}[Intertwiner]
If $H,H':A\rightsquigarrow B$ are correspondences, an intertwiner $T:H\to H'$ is a bounded linear map such that for all $a\in A$ and $b\in B$ the following diagrams commute:

% https://q.uiver.app/#q=WzAsNCxbMCwwLCJIIl0sWzEsMCwiSCciXSxbMCwxLCJIIl0sWzEsMSwiSCciXSxbMCwyLCJcXGxhbWJkYV97SH0oYSkiLDJdLFsxLDMsIlxcbGFtYmRhX3tIJ30oYSkiXSxbMCwxLCJUIl0sWzIsMywiVCIsMl1d
\[\begin{tikzcd}
	H & {H'} \\
	H & {H'}
	\arrow["T", from=1-1, to=1-2]
	\arrow["{\lambda_{H}(a)}"', from=1-1, to=2-1]
	\arrow["{\lambda_{H'}(a)}", from=1-2, to=2-2]
	\arrow["T"', from=2-1, to=2-2]
\end{tikzcd} \quad \text{and}\quad\begin{tikzcd}
	H & {H'} \\
	H & {H'}
	\arrow["T", from=1-1, to=1-2]
	\arrow["{\rho_{H}(b^{\mathrm{op}})}"', from=1-1, to=2-1]
	\arrow["{\rho_{H'}(b^{\mathrm{op}})}", from=1-2, to=2-2]
	\arrow["T"', from=2-1, to=2-2]
\end{tikzcd}\]
\end{definition}

Equivalently,
$
T(a\cdot \xi \cdot b)=a\cdot T(\xi)\cdot b
\quad (a\in A,\ b\in B).
$
Intertwiners form complex vector spaces.

\begin{definition}[Standard form and $L^2(A)$]
For each $A\in V$, choose its Haagerup standard form $(H_A,\lambda_A,J_A,P_A)$. Define the $A$-$A$ correspondence
$$
L^2(A):=H_A
$$
with left action $\lambda_{L^2(A)}:=\lambda_A$ and right action defined by
$$
\rho_{L^2(A)}(a^{\mathrm{op}}):=J_A\,\lambda_A(a^*)\,J_A
\qquad
(a\in A).
$$
The left and right actions commute, and $L^2(A)$ is the horizontal identity correspondence at $A$.

It is useful to keep the two commuting actions visible as a boundary diagram:
 % https://q.uiver.app/#q=WzAsNCxbMCwwLCJBIl0sWzEsMCwiQihMXjIoQSkpIl0sWzAsMSwiQSJdLFsxLDEsIkIoTF4yKEEpKSJdLFsxLDMsIlxcdGV4dHtpZH0iXSxbMCwyLCJcXHRleHR7aWR9IiwyXSxbMCwxLCJcXGxhbWJkYV97QX0iXSxbMiwzLCJcXHJob197TF57Mn0oQSl9IiwyXV0=
\[\begin{tikzcd}
	A & {B(L^2(A))} \\
	A & {B(L^2(A))}
	\arrow["{\lambda_{A}}", from=1-1, to=1-2]
	\arrow["{\text{id}}"', from=1-1, to=2-1]
	\arrow["{\text{id}}", from=1-2, to=2-2]
	\arrow["{\rho_{L^{2}(A)}}"', from=2-1, to=2-2]
\end{tikzcd}\quad \text{with}\quad \lambda_A(a)\,\rho_{L^2(A)}(b^{\mathrm{op}})=\rho_{L^2(A)}(b^{\mathrm{op}})\,\lambda_A(a).\]
\end{definition}

\begin{definition}[Connes fusion]
If $H:A\rightsquigarrow B$ and $K:B\rightsquigarrow C$, their horizontal composite is the Connes fusion
$$
H\boxtimes_B K: A\rightsquigarrow C,
$$
the standard relative tensor product of correspondences over $B$. Connes fusion is associative and unital up to canonical isomorphism, with unit $L^2(B)$. This is the source of the pseudo coherence in the horizontal direction.
\end{definition}

\textit{Remark.} No linear structure is assumed on $\mathbf{vN}_{\mathrm{vert}}$ itself. The linearity used below is that spaces of squares (intertwiners) are complex vector spaces and the double-category operations are linear or bilinear on them.

Write $Q_{\mathrm{lin}}$ for the free globularly generated construction in the linear/enriched setting appropriate to correspondences and intertwiners \cite{orendain2021freeglobularlygenerateddouble}.

Now, we have all the components to define our target (pseudo) double category:

\begin{definition}
Let $\widetilde{\mathbb{W}}^\ast$ be a saturated decorated bicategory in the operator-algebraic framework whose decorated horizontalization encodes:
\begin{enumerate}[label=\alph*)]
\item objects $V$,
\item vertical morphisms $\mathbf{vN}_{\mathrm{vert}}$,
\item horizontal morphisms correspondences,
\item horizontal identities $A\mapsto L^2(A)$,
\item horizontal composition Connes fusion.
\end{enumerate}
Define
$$
\mathbf{vNA} := Q_{\mathrm{lin}}\bigl(\widetilde{\mathbb{W}}^\ast\bigr).
$$
Then $\mathbf{vNA}$ is a globularly generated pseudo double category with vertical category $\mathbf{vN}_{\mathrm{vert}}$, horizontal morphisms given by correspondences, and squares given by intertwiners.
\end{definition}

For every vertical morphism $\varphi:A\to B$ in $\mathbf{vN}_{\mathrm{vert}}$, the horizontal-identity functor
$
U_{\mathbf{vNA}}:\mathbf{vNA}_0 \to \mathbf{vNA}_1
$
yields a canonical unit square
$$
\iota_\varphi := U_{\mathbf{vNA}}(\varphi): L^2(A) \longrightarrow L^2(B),
$$
with vertical boundary $(\varphi,\varphi)$ and horizontal source and target the identity correspondences $L^2(A)$ and $L^2(B)$. The underlying intertwiner is the standard-form map
$$
L^2(\varphi):L^2(A)\longrightarrow L^2(B).
$$ 

Pictorially, the unit square is a square in $\mathbf{vNA}_{1}$

% https://q.uiver.app/#q=WzAsNCxbMCwwLCJBIl0sWzEsMCwiQSJdLFswLDEsIkIiXSxbMSwxLCJCIl0sWzEsMywiXFxwaGkiXSxbMCwyLCJcXHBoaSIsMl0sWzAsMSwiTF4yKEEpIl0sWzIsMywiTF4yKEIpIiwyXV0=
\[\begin{tikzcd}
	A & A \\
	B & B
	\arrow["{L^2(A)}", from=1-1, to=1-2]
	\arrow["\varphi"', from=1-1, to=2-1]
	\arrow["\varphi", from=1-2, to=2-2]
	\arrow["{L^2(B)}"', from=2-1, to=2-2]
\end{tikzcd} \quad \text{with interior}\quad \iota_\varphi: L^2(A) \Rightarrow L^2(B).\]

and the interior condition is that the bimodularity diagrams commute with respect to the induced $A$- and $B$-actions.

The unit square construction always produces a square whose vertical boundary is $(\varphi,\varphi)$ and whose horizontal edges are identity correspondences. In the AQFT double functor, a commuting spacetime square

% https://q.uiver.app/#q=WzAsNCxbMCwxLCJVJyJdLFsxLDEsIlYnIl0sWzAsMCwiVSJdLFsxLDAsIlYiXSxbMCwxLCJoJyIsMl0sWzIsMywiaCJdLFsyLDAsImkiLDJdLFszLDEsImoiXV0=
\[\begin{tikzcd}
	U & V \\
	{U'} & {V'}
	\arrow["h", from=1-1, to=1-2]
	\arrow["i"', from=1-1, to=2-1]
	\arrow["j", from=1-2, to=2-2]
	\arrow["{h'}"', from=2-1, to=2-2]
\end{tikzcd}\quad (j \circ h = h' \circ i)\]

yields vertical maps
$
\mathcal{A}(i):\mathcal{A}(U)\longrightarrow \mathcal{A}(U')$ and $
\mathcal{A}(j):\mathcal{A}(V)\longrightarrow \mathcal{A}(V'),
$
and horizontal edges given by correspondences built from $L^2(\mathcal{A}(V))$ and $L^2(\mathcal{A}(V'))$ with left actions restricted along $\mathcal{A}(h)$ and $\mathcal{A}(h')$. The assigned square uses the operator $L^2(\mathcal{A}(j))$, with boundary determined by
$
\mathcal{A}(j)\circ \mathcal{A}(h)=\mathcal{A}(h')\circ \mathcal{A}(i),
$
which is exactly what makes $L^2(\mathcal{A}(j))$ bimodular for the restricted actions (proved explicitly below).

\textbf{AQFT Double Functor.} The AQFT input is a functor on admissible embeddings landing in the vertical morphisms $\mathbf{vN}_{\mathrm{vert}}$. The double functor $\mathcal{F}_{\mathcal{A}}$ refines this by assigning to each embedding a correspondence built from standard form, and to each commuting spacetime square an intertwiner whose source/target correspondences and vertical boundary maps are the ones induced by that square.

Fix a functor
$
\mathcal{A}:\mathbf{Emb}(M)\longrightarrow \mathbf{vN}_{\mathrm{vert}}
$
such that:
\begin{enumerate}[label=\alph*)]
\item for every inclusion $i:U\hookrightarrow V$ in $\mathbf{Inc}(M)$, the morphism
$$
\mathcal{A}(i):\mathcal{A}(U)\longrightarrow \mathcal{A}(V)
$$
is injective,
\item $\mathcal{A}$ preserves identities and composition on admissible embeddings.
\end{enumerate}
The condition $\mathcal{A}(\mathbf{Emb}(M))\subseteq \mathbf{vN}_{\mathrm{vert}}$ ensures that for each admissible embedding $j:V\hookrightarrow V'$ the standard-form map
$$
L^2\bigl(\mathcal{A}(j)\bigr):L^2\bigl(\mathcal{A}(V)\bigr)\longrightarrow L^2\bigl(\mathcal{A}(V')\bigr)
$$
is available functorially in the target double category.

\begin{definition}[AQFT double functor $\mathcal{F}_{\mathcal{A}}$]
Define a double functor
$$
\mathcal{F}_{\mathcal{A}}:\mathbf{Mink}(M)\longrightarrow \mathbf{vNA}
$$
by the following assignments.
\begin{enumerate}[label=(\alph*)]
\item Objects. For a region $U$, set
$$
\mathcal{F}_{\mathcal{A}}(U):=\mathcal{A}(U).
$$

\item Vertical morphisms. For an inclusion $i:U\hookrightarrow V$, set
$$
\mathcal{F}_{\mathcal{A}}(i):=\mathcal{A}(i):\mathcal{A}(U)\to \mathcal{A}(V).
$$

\item Horizontal morphisms. For an admissible embedding $h:U\hookrightarrow V$, define $\mathcal{F}_{\mathcal{A}}(h)$ to be the correspondence
$$
\mathcal{F}_{\mathcal{A}}(h):\mathcal{A}(U)\rightsquigarrow \mathcal{A}(V)
$$
with underlying Hilbert space $L^2(\mathcal{A}(V))$, right $\mathcal{A}(V)$-action the standard right action on $L^2(\mathcal{A}(V))$, and left $\mathcal{A}(U)$-action given by restriction of scalars along
$$
\mathcal{A}(h):\mathcal{A}(U)\to \mathcal{A}(V).
$$
Concretely, if
$$
\lambda_V:\mathcal{A}(V)\to B\bigl(L^2(\mathcal{A}(V))\bigr)
$$
denotes the standard left representation, then the left action of $\mathcal{A}(U)$ on $\mathcal{F}_{\mathcal{A}}(h)$ is
$$
\lambda_h:=\lambda_V\circ \mathcal{A}(h):\mathcal{A}(U)\to B\bigl(L^2(\mathcal{A}(V))\bigr).
$$

\item Squares. Given a square in $\mathbf{Mink}(M)$, i.e. a commuting diagram in $\mathbf{Emb}(M)$ with inclusion verticals,
% https://q.uiver.app/#q=WzAsNCxbMCwxLCJVJyJdLFsxLDEsIlYnIl0sWzAsMCwiVSJdLFsxLDAsIlYiXSxbMCwxLCJoJyIsMl0sWzIsMywiaCJdLFsyLDAsImkiLDJdLFszLDEsImoiXV0=
\[\begin{tikzcd}
	U & V \\
	{U'} & {V'}
	\arrow["h", from=1-1, to=1-2]
	\arrow["i"', from=1-1, to=2-1]
	\arrow["j", from=1-2, to=2-2]
	\arrow["{h'}"', from=2-1, to=2-2]
\end{tikzcd}\quad \text{with} \quad j \circ h = h' \circ i\]
define $\mathcal{F}_{\mathcal{A}}$ on such a square to be the square in $\mathbf{vNA}$
$$
\mathcal{F}_{\mathcal{A}}(\Box):\mathcal{F}_{\mathcal{A}}(h)\Longrightarrow \mathcal{F}_{\mathcal{A}}(h')
$$
whose vertical boundary is $\bigl(\mathcal{A}(i),\mathcal{A}(j)\bigr)$ and whose underlying intertwiner is the standard-form operator
$$
L^2\bigl(\mathcal{A}(j)\bigr):L^2\bigl(\mathcal{A}(V)\bigr)\longrightarrow L^2\bigl(\mathcal{A}(V')\bigr).
$$
Equivalently, the assigned square is displayed as
% https://q.uiver.app/#q=WzAsNCxbMCwwLCJcXG1hdGhjYWx7QX0oVSkiXSxbMSwwLCJcXG1hdGhjYWx7QX0oVikiXSxbMCwxLCJcXG1hdGhjYWx7QX0oVScpIl0sWzEsMSwiXFxtYXRoY2Fse0F9KFYnKSJdLFswLDEsIlxcbWF0aGNhbHtGfV97XFxtYXRoY2Fse0F9fShoKSJdLFsyLDMsIlxcbWF0aGNhbHtGfV97XFxtYXRoY2Fse0F9fShoJykiLDJdLFsxLDMsIlxcbWF0aGNhbHtBfShqKSJdLFswLDIsIlxcbWF0aGNhbHtBfShpKSIsMl1d
\[\begin{tikzcd}
	{\mathcal{A}(U)} & {\mathcal{A}(V)} \\
	{\mathcal{A}(U')} & {\mathcal{A}(V')}
	\arrow["{\mathcal{F}_{\mathcal{A}}(h)}", from=1-1, to=1-2]
	\arrow["{\mathcal{A}(i)}"', from=1-1, to=2-1]
	\arrow["{\mathcal{A}(j)}", from=1-2, to=2-2]
	\arrow["{\mathcal{F}_{\mathcal{A}}(h')}"', from=2-1, to=2-2]
\end{tikzcd}\]
with interior map $L^2(\mathcal{A}(j))$.

Bimodularity for the restricted left actions follows from functoriality,
$
\mathcal{A}(j)\circ \mathcal{A}(h)=\mathcal{A}(h')\circ \mathcal{A}(i),
$
and is verified explicitly in the square-typing check.
\end{enumerate}
\end{definition}

\textit{Remark.} The operator $L^2(\mathcal{A}(j))$ arises functorially from the unit-square mechanism in $\mathbf{vNA}$, but the square in (d) is not the unit square $U_{\mathbf{vNA}}(\mathcal{A}(j))$. The unit square has vertical boundary $\bigl(\mathcal{A}(j),\mathcal{A}(j)\bigr)$ and horizontal edges the identity correspondences $L^2(\mathcal{A}(V))$ and $L^2(\mathcal{A}(V'))$, whereas the square in (d) has vertical boundary $\bigl(\mathcal{A}(i),\mathcal{A}(j)\bigr)$ and horizontal edges the restricted correspondences $\mathcal{F}_{\mathcal{A}}(h)$ and $\mathcal{F}_{\mathcal{A}}(h')$. The underlying operator is the same $L^2(\mathcal{A}(j))$, but it is viewed as an intertwiner for different module structures.

\textbf{Functoriality.} Defining $\mathcal{F}_{\mathcal{A}}$ on squares is the only nontrivial typing point: a commuting spacetime square must determine a square in $\mathbf{vNA}$, hence an intertwiner compatible with the bimodule structures of the horizontal correspondences. Here the interior operator is forced to be $T=L^2(\mathcal{A}(j))$, and the commutativity $\mathcal{A}(j)\circ \mathcal{A}(h)=\mathcal{A}(h')\circ \mathcal{A}(i)$ is precisely what makes $T$ bimodular for the restricted left actions. We now write this out explicitly.

Fix a commuting square in $\mathbf{Mink}(M)$:

% https://q.uiver.app/#q=WzAsNCxbMCwwLCJVIl0sWzEsMCwiViJdLFswLDEsIlUnIl0sWzEsMSwiViciXSxbMCwxLCJoIl0sWzIsMywiaCciLDJdLFsxLDMsImoiXSxbMCwyLCJpIiwyXV0=
\[\begin{tikzcd}
	U & V \\
	{U'} & {V'}
	\arrow["h", from=1-1, to=1-2]
	\arrow["i"', from=1-1, to=2-1]
	\arrow["j", from=1-2, to=2-2]
	\arrow["{h'}"', from=2-1, to=2-2]
\end{tikzcd}\quad \text{so that}\quad j \circ h = h' \circ i.\]

Applying $\mathcal{A}:\mathbf{Emb}(M)\to \mathbf{vN}_{\mathrm{vert}}$ gives a commuting square of normal unital $*$-homomorphisms:

% https://q.uiver.app/#q=WzAsNCxbMCwwLCJcXG1hdGhjYWx7QX0oVSkiXSxbMSwwLCJcXG1hdGhjYWx7QX0oVikiXSxbMCwxLCJcXG1hdGhjYWx7QX0oVScpIl0sWzEsMSwiXFxtYXRoY2Fse0F9KFYnKSJdLFswLDEsIlxcbWF0aGNhbHtBfShoKSJdLFsyLDMsIlxcbWF0aGNhbHtBfShoJykiLDJdLFsxLDMsIlxcbWF0aGNhbHtBfShqKSJdLFswLDIsIlxcbWF0aGNhbHtBfShpKSIsMl1d
\[\begin{tikzcd}
	{\mathcal{A}(U)} & {\mathcal{A}(V)} \\
	{\mathcal{A}(U')} & {\mathcal{A}(V')}
	\arrow["{\mathcal{A}(h)}", from=1-1, to=1-2]
	\arrow["{\mathcal{A}(i)}"', from=1-1, to=2-1]
	\arrow["{\mathcal{A}(j)}", from=1-2, to=2-2]
	\arrow["{\mathcal{A}(h')}"', from=2-1, to=2-2]
\end{tikzcd} \quad \text{so that}\quad \mathcal{A}(j)\circ \mathcal{A}(h)=\mathcal{A}(h')\circ \mathcal{A}(i).\]

Set
$$
H:=\mathcal{F}_{\mathcal{A}}(h):\mathcal{A}(U)\rightsquigarrow \mathcal{A}(V),
\quad
H':=\mathcal{F}_{\mathcal{A}}(h'):\mathcal{A}(U')\rightsquigarrow \mathcal{A}(V').
$$
Then the underlying Hilbert spaces are
$$
H=L^2\bigl(\mathcal{A}(V)\bigr),
\quad
H'=L^2\bigl(\mathcal{A}(V')\bigr).
$$

Write $\lambda_V,\rho_V$ for the standard left and right actions of $\mathcal{A}(V)$ on $L^2(\mathcal{A}(V))$, and $\lambda_{V'},\rho_{V'}$ similarly for $\mathcal{A}(V')$. The right actions are the standard ones,
$$
\xi\cdot_H b:=\rho_V(b)\xi,
\quad
\eta\cdot_{H'} b':=\rho_{V'}(b')\eta,
$$
while the left actions are obtained by restriction of scalars along the embedding maps,
$$
a\cdot_H \xi:=\lambda_V\bigl(\mathcal{A}(h)(a)\bigr)\xi \quad (a\in \mathcal{A}(U)),
\quad
a'\cdot_{H'} \eta:=\lambda_{V'}\bigl(\mathcal{A}(h')(a')\bigr)\eta \quad (a'\in \mathcal{A}(U')).
$$

The square $\mathcal{F}_{\mathcal{A}}(\Box)$ is defined to have interior operator
$$
T:=L^2\bigl(\mathcal{A}(j)\bigr):L^2\bigl(\mathcal{A}(V)\bigr)\longrightarrow L^2\bigl(\mathcal{A}(V')\bigr).
$$
This is defined since $\mathcal{A}(j)\in \mathbf{vN}_{\mathrm{vert}}$ and the standard-form assignment is functorial on $\mathbf{vN}_{\mathrm{vert}}$. In particular, for all $x,b\in \mathcal{A}(V)$ and $\xi\in L^2\bigl(\mathcal{A}(V)\bigr)$,
$$
T\bigl(\lambda_V(x)\xi\bigr)=\lambda_{V'}\bigl(\mathcal{A}(j)(x)\bigr)\,T(\xi),
\quad
T(\xi\cdot b)=T(\xi)\cdot \mathcal{A}(j)(b).
$$
We will use these operator-algebraic identities again.

At this stage the data are fixed. It remains to check that these standard intertwining identities imply the bimodularity identities for $T$ with respect to the \emph{restricted} module structures defining $H$ and $H'$.

Let $b\in \mathcal{A}(V)$ and $\xi\in L^2\bigl(\mathcal{A}(V)\bigr)$. Since the right actions on $H$ and $H'$ are the standard ones, the right intertwining identity gives
$$
T(\xi\cdot_H b)=T(\xi\cdot b)=T(\xi)\cdot \mathcal{A}(j)(b)=T(\xi)\cdot_{H'} \mathcal{A}(j)(b).
$$
Thus $T$ has right boundary map $\mathcal{A}(j)$.

Let $a\in \mathcal{A}(U)$ and $\xi\in L^2\bigl(\mathcal{A}(V)\bigr)$. By definition of the restricted left action on $H$,
$$
T(a\cdot_H \xi)=T\bigl(\lambda_V(\mathcal{A}(h)(a))\,\xi\bigr).
$$
Applying the left intertwining identity gives
$$
T\bigl(\lambda_V(\mathcal{A}(h)(a))\,\xi\bigr)
=
\lambda_{V'}\bigl(\mathcal{A}(j)(\mathcal{A}(h)(a))\bigr)\,T(\xi).
$$
Using commutativity $\mathcal{A}(j)\circ \mathcal{A}(h)=\mathcal{A}(h')\circ \mathcal{A}(i)$, this becomes
$$
T(a\cdot_H \xi)=\lambda_{V'}\bigl(\mathcal{A}(h')(\mathcal{A}(i)(a))\bigr)\,T(\xi).
$$
Recognizing the restricted left action on $H'$, we obtain
$$
T(a\cdot_H \xi)=\mathcal{A}(i)(a)\cdot_{H'} T(\xi),
$$
which is the required left-module compatibility with left boundary map $\mathcal{A}(i)$.

The preceding identities show that $T=L^2(\mathcal{A}(j))$, regarded as an operator $H=\mathcal{F}_{\mathcal{A}}(h)\to H'=\mathcal{F}_{\mathcal{A}}(h')$, is bimodular for the induced boundary maps, namely left boundary $\mathcal{A}(i)$ and right boundary $\mathcal{A}(j)$. Equivalently, this assignment sends each commuting spacetime square to a well-typed square in $\mathbf{vNA}$, so $\mathcal{F}_{\mathcal{A}}$ is well-defined on squares.

\begin{proposition}
Therefore, $\mathcal{F}_{\mathcal{A}}$ defines a pseudo double functor
$$
\mathcal{F}_{\mathcal{A}}:\mathbf{Mink}(M)\longrightarrow \mathbf{vNA}.
$$
\end{proposition}

\begin{proof}
On objects and vertical morphisms, $\mathcal{F}_{\mathcal{A}}$ agrees with $\mathcal{A}$ (restricted to inclusions), hence preserves vertical identities and composition by functoriality of $\mathcal{A}$.

For each region $U$, the horizontal identity is sent to the horizontal unit in $\mathbf{vNA}$,
$$
\mathcal{F}_{\mathcal{A}}(\mathrm{id}_U)=L^2\bigl(\mathcal{A}(U)\bigr):\mathcal{A}(U)\rightsquigarrow \mathcal{A}(U).
$$
For composable admissible embeddings $U\xhookrightarrow{h}V\xhookrightarrow{k}W$, the correspondences
$\mathcal{F}_{\mathcal{A}}(h)$ and $\mathcal{F}_{\mathcal{A}}(k)$ are given by $L^2\bigl(\mathcal{A}(V)\bigr)$ and $L^2\bigl(\mathcal{A}(W)\bigr)$ with left actions restricted along $\mathcal{A}(h)$ and $\mathcal{A}(k)$, respectively. By the defining structure of $\mathbf{vNA}$ (Connes fusion with unit $L^2(\mathcal{A}(V))$ and compatibility with restriction of scalars), there is a canonical unitary isomorphism
$$
\mathcal{F}_{\mathcal{A}}(h)\boxtimes_{\mathcal{A}(V)}\mathcal{F}_{\mathcal{A}}(k)\cong \mathcal{F}_{\mathcal{A}}(k\circ h),
$$,
where the left action on the right-hand side is restricted along
$\mathcal{A}(k\circ h)=\mathcal{A}(k)\circ \mathcal{A}(h)$.

On squares, well-typedness holds: for a commuting spacetime square with right side $j$, the assigned operator is $L^2\bigl(\mathcal{A}(j)\bigr)$ and is bimodular for the induced boundary actions. Compatibility with vertical pasting follows from functoriality of $L^2$,
$$
L^2\bigl(\mathcal{A}(j_2)\bigr)\circ L^2\bigl(\mathcal{A}(j_1)\bigr)
=
L^2\bigl(\mathcal{A}(j_2\circ j_1)\bigr).
$$
Compatibility with horizontal pasting is the fusion-compatibility built into $\mathbf{vNA}$ (and the choice of $\mathbf{vN}_{\mathrm{vert}}$). Therefore the pseudo double functor axioms are satisfied.
\end{proof}

We have thus constructed $\mathcal{F}_{\mathcal{A}}$ from the AQFT input $\mathcal{A}$. Next we extract the underlying net $\mathcal{A}_{\mathcal{F}_{\mathcal{A}}}$ and state the Haag--Kastler axioms in this language.

We now pass from the double-functorial data to the underlying net of local algebras by restricting $\mathcal{A}$ to inclusions. This restricted functor is the Haag--Kastler net on which we state the axioms:

\begin{definition}[Haag--Kastler net]
Restrict $\mathcal{A}$ to inclusions to obtain a functor
$$
\mathcal{A}_{\mathcal{F}_{\mathcal{A}}}:\mathbf{Inc}(M)\longrightarrow \mathbf{vN}_{\mathrm{vert}},
\quad
\mathcal{A}_{\mathcal{F}_{\mathcal{A}}}(U)=\mathcal{A}(U),
\quad
\mathcal{A}_{\mathcal{F}_{\mathcal{A}}}(i)=\mathcal{A}(i).
$$
This is the net to which the Haag--Kastler axioms are applied.
\end{definition}

\textbf{Haag-Kastler Axioms.} We state the Haag--Kastler axioms: isotony, locality, covariance and time-slice in our double functorial formulation.

\begin{axiom}[\textbf{HK1 (Isotony)}]
For every inclusion $i:U\hookrightarrow V$, the morphism
$$
\mathcal{A}_{\mathcal{F}_{\mathcal{A}}}(i):\mathcal{A}_{\mathcal{F}_{\mathcal{A}}}(U)\longrightarrow \mathcal{A}_{\mathcal{F}_{\mathcal{A}}}(V)
$$
is an injective normal unital $*$-homomorphism.
\end{axiom}

\begin{claim}
If $\mathcal{A}(i)$ is injective for every inclusion $i$, then $\mathcal{A}_{\mathcal{F}_{\mathcal{A}}}$ satisfies HK1.
\end{claim}

\begin{proof}
By definition of the extracted net,
$$
\mathcal{A}_{\mathcal{F}_{\mathcal{A}}}(i)=\mathcal{A}(i),
\quad
\mathcal{A}_{\mathcal{F}_{\mathcal{A}}}(U)=\mathcal{A}(U).
$$
Hence $\mathcal{A}_{\mathcal{F}_{\mathcal{A}}}(i)$ is injective whenever $\mathcal{A}(i)$ is.
\end{proof}

\begin{lemma}[Functorial isotony]
If $U\subseteq V\subseteq W$, then
$$
\mathcal{A}_{\mathcal{F}_{\mathcal{A}}}(U\hookrightarrow W)
=
\mathcal{A}_{\mathcal{F}_{\mathcal{A}}}(V\hookrightarrow W)\circ
\mathcal{A}_{\mathcal{F}_{\mathcal{A}}}(U\hookrightarrow V).
$$  
\end{lemma}

\begin{proof}
In $\mathbf{Inc}(M)$, write $i_{UV}:U\hookrightarrow V$, $i_{VW}:V\hookrightarrow W$, and $i_{UW}:U\hookrightarrow W$, so that $i_{UW}=i_{VW}\circ i_{UV}$. Since $\mathcal{A}_{\mathcal{F}_{\mathcal{A}}}:\mathbf{Inc}(M)\to \mathbf{vN}_{\mathrm{vert}}$ is a functor,
$$
\mathcal{A}_{\mathcal{F}_{\mathcal{A}}}(i_{UW})
=
\mathcal{A}_{\mathcal{F}_{\mathcal{A}}}(i_{VW}\circ i_{UV})
=
\mathcal{A}_{\mathcal{F}_{\mathcal{A}}}(i_{VW})\circ \mathcal{A}_{\mathcal{F}_{\mathcal{A}}}(i_{UV}).
$$
\end{proof}

Now we move on to locality. Locality is formulated using an orthogonality (causal disjointness) relation on inclusions into a common region. Since $\mathbf{Mink}(M)$ does not encode independence by additional generating squares, locality is imposed as an axiom on the extracted net.

Fix a symmetric relation $\perp$ on pairs of inclusions with common codomain, stable under postcomposition.

\begin{axiom}[\textbf{HK2 (Locality)}]
If $i_U:U\hookrightarrow T$ and $i_V:V\hookrightarrow T$ satisfy $i_U\perp i_V$, then the represented subalgebras commute in $B\bigl(L^2(\mathcal{A}(T))\bigr)$: for all
$a\in \mathcal{A}(U)$ and $b\in \mathcal{A}(V)$,
$$
\bigl[\lambda_T\bigl(\mathcal{A}(i_U)(a)\bigr),\ \lambda_T\bigl(\mathcal{A}(i_V)(b)\bigr)\bigr]=0,
$$
where $\lambda_T:\mathcal{A}(T)\to B\bigl(L^2(\mathcal{A}(T))\bigr)$ is the faithful standard-form left representation.
\end{axiom}

\begin{claim}
Under HK2, the net $\mathcal{A}_{\mathcal{F}_{\mathcal{A}}}$ is local: if $i_U\perp i_V$, then the subalgebras $\mathcal{A}(i_U)\bigl(\mathcal{A}(U)\bigr)$ and $\mathcal{A}(i_V)\bigl(\mathcal{A}(V)\bigr)$ commute inside $\mathcal{A}(T)$.
\end{claim}

\begin{proof}
Take $x\in \mathcal{A}(i_U)\bigl(\mathcal{A}(U)\bigr)$ and $y\in \mathcal{A}(i_V)\bigl(\mathcal{A}(V)\bigr)$. HK2 gives
$$
[\lambda_T(x),\lambda_T(y)]=0,
$$
hence
$$
\lambda_T([x,y])=0.
$$
Since $\lambda_T$ is faithful, $[x,y]=0$ in $\mathcal{A}(T)$.
\end{proof}

We now define covariance. Let $\mathbf{Iso}(M)\subseteq \mathbf{Emb}(M)$ be the groupoid of admissible isomorphisms, and restrict $\mathcal{A}$ to $\mathbf{Iso}(M)$.

\begin{axiom}[\textbf{HK3 (Covariance)}]
The restriction $\mathcal{A}\!\restriction_{\mathbf{Iso}(M)}$ lands in $*$-isomorphisms and is functorial on the groupoid.
\end{axiom}

\begin{claim}
Under HK3, the extracted net is covariant: for each isomorphism $g:U\to V$ in $\mathbf{Iso}(M)$, the map
$$
r_g:=\mathcal{A}(g):\mathcal{A}(U)\longrightarrow \mathcal{A}(V)
$$
is a $*$-isomorphism, and for composable $g_1:U\to V$ and $g_2:V\to W$ one has
$$
r_{g_2\circ g_1}=r_{g_2}\circ r_{g_1}.
$$
\end{claim}

\begin{proof}
Functoriality of $\mathcal{A}$ on the groupoid $\mathbf{Iso}(M)$ gives
$
\mathcal{A}(g_2\circ g_1)=\mathcal{A}(g_2)\circ \mathcal{A}(g_1),
$
which is exactly the displayed coherence identity.
\end{proof}

Since $\mathcal{A}_{\mathcal{F}_{\mathcal{A}}}$ is indexed by inclusions, time-slice is formulated for a chosen class of Cauchy inclusions. Let $S$ be a specified class of inclusions $i:U\hookrightarrow V$ such that $U$ contains a Cauchy surface for $V$.

\begin{axiom}[\textbf{HK4 (Time-slice)}]
If $i\in S$, then $\mathcal{A}(i):\mathcal{A}(U)\to \mathcal{A}(V)$ is an isomorphism.
\end{axiom}

\begin{claim}
Under HK4, the extracted net satisfies time-slice: if $i:U\hookrightarrow V$ lies in $S$, then
$
\mathcal{A}_{\mathcal{F}_{\mathcal{A}}}(U)\cong \mathcal{A}_{\mathcal{F}_{\mathcal{A}}}(V).
$
\end{claim}

\begin{proof}
By definition of the extracted net, $\mathcal{A}_{\mathcal{F}_{\mathcal{A}}}(i)=\mathcal{A}(i)$. Thus $\mathcal{A}_{\mathcal{F}_{\mathcal{A}}}(i)$ is an isomorphism whenever $\mathcal{A}(i)$ is.
\end{proof}

\begin{axiom}[\textbf{HK5 (Additivity)}]
For a cover $U=\bigcup_\alpha U_\alpha$ by admissible subregions with inclusions $i_\alpha:U_\alpha\hookrightarrow U$,
$$
\mathcal{A}(U)=\bigvee_\alpha \,\mathcal{A}(i_\alpha)\bigl(\mathcal{A}(U_\alpha)\bigr),
$$
where $\bigvee$ denotes the von Neumann algebra generated inside $\mathcal{A}(U)$ by the indicated subalgebras.
\end{axiom}

We now turn to structural properties of the double functor and the underlying net.

\section{Structural Properties of the AQFT Double Functor}\label{sec:structural}

This section proves a few structural statements whose clean formulation uses the square-level data of
$$
\mathcal{F}_{\mathcal{A}}:\mathbf{Mink}(M)\longrightarrow \mathbf{vNA},
$$
together with basic properties of the extracted Haag-Kastler net
$
\mathcal{A}_{\mathcal{F}_{\mathcal{A}}}=\mathcal{A}\!\restriction_{\mathbf{Inc}(M)}.
$
Throughout, $\mathbf{Mink}(M)$, $\mathbf{vNA}$, the AQFT input $\mathcal{A}:\mathbf{Emb}(M)\to \mathbf{vN}_{\mathrm{vert}}$, and the construction of $\mathcal{F}_{\mathcal{A}}$ are as in \S\ref{sec:double}: horizontals are given by restriction of scalars on $L^2$ and a spacetime square with right vertical map $j$ is carried by the operator $L^2\bigl(\mathcal{A}(j)\bigr)$, viewed as an intertwiner for the induced boundary actions.

\textbf{Uniqueness from $\gamma$-generation.}
As already noted, $\mathbf{Mink}(M)$ is generally not globularly generated in a geometrically interesting way: in fact $G\bigl(\mathbf{Mink}(M)\bigr)$ is essentially just identity squares on horizontal morphisms. Still, the $\gamma$-construction is a handy bookkeeping device for uniqueness-from-generators statements about squares.

Set
$$
\mathbf{Mink}_\gamma(M):=\gamma\bigl(\mathbf{Mink}(M)\bigr).
$$
Concretely, $\mathbf{Mink}_\gamma(M)$ has the same objects, vertical morphisms, and horizontal morphisms as $\mathbf{Mink}(M)$, and its squares form the smallest class of squares containing all globular squares and all unit squares, and closed under horizontal and vertical composition (hence under whiskering).

Recall: for a pseudo double category $\mathbb{C}$, write $G(\mathbb{C})$ for its class of globular squares, and for each vertical morphism $f\in \mathrm{Mor}(\mathbb{C}_0)$ write $\iota_f:=U(f)$ for the unit square. Then the $0$-marked squares are
$$
G_{\mathbb{C}}
:=
G(\mathbb{C}) \,\cup\, \{\iota_f \mid f\in \mathrm{Mor}(\mathbb{C}_0)\}.
$$

Whiskering is just composition with identity squares, so it is automatically present once one has closure under horizontal composition. Concretely, if $\alpha:H\Rightarrow K$ is a square with vertical boundary $(f,g)$ and $R$ is horizontally composable on the right, then the right whisker $\alpha\odot R$ is the pasted square obtained by horizontally composing $\alpha$ with the identity square on $R$ (and similarly on the left).

\begin{proposition}[Uniqueness from generators on $\mathbf{Mink}_\gamma(M)$]\label{prop:gamma-uniq}
Let $F,G:\mathbf{Mink}_\gamma(M)\to \mathbf{vNA}$ be double functors. Assume:
\begin{enumerate}[label=\alph*)]
\item $F$ and $G$ agree on objects, on vertical morphisms, and on horizontal morphisms; and
\item $F(\sigma)=G(\sigma)$ for every generating square $\sigma\in G_{\mathbf{Mink}(M)}$ (globular squares and unit squares).
\end{enumerate}
Then $F=G$ on all squares, hence $F=G$ as double functors.
\end{proposition}

\begin{proof}
Let $S$ be the class of squares $\alpha$ in $\mathbf{Mink}_\gamma(M)$ such that $F(\alpha)=G(\alpha)$.
By assumption, $S$ contains $G_{\mathbf{Mink}(M)}$.

Since $F$ and $G$ are double functors, they preserve vertical composition of squares and horizontal composition of squares. Hence $S$ is closed under both vertical and horizontal composition: if $\alpha,\beta\in S$ are composable vertically (resp.\ horizontally), then
$$
F(\beta\circ \alpha)=F(\beta)\circ F(\alpha)=G(\beta)\circ G(\alpha)=G(\beta\circ \alpha),
$$
$$
F(\beta\odot \alpha)=F(\beta)\odot F(\alpha)=G(\beta)\odot G(\alpha)=G(\beta\odot \alpha).
$$
By definition of $\gamma\bigl(\mathbf{Mink}(M)\bigr)$, every square of $\mathbf{Mink}_\gamma(M)$ lies in the closure of $G_{\mathbf{Mink}(M)}$ under vertical and horizontal composition. Therefore every square lies in $S$, i.e.\ $F$ and $G$ agree on all squares.
\end{proof}

\textit{Remark.}
Proposition \ref{prop:gamma-uniq} is only a uniqueness-from-generators statement about the \emph{square assignment}. We use it as a bookkeeping tool (especially when a construction is forced on globular/unit squares and then extended by pasting), not as a claim that $\mathbf{Mink}_\gamma(M)$ captures spacetime geometry.

\textbf{(Co)limits inside an ambient algebra: unions as joins under additivity.}
In practice one rarely takes categorical pushouts/pullbacks in $\mathbf{vN}$. Instead one works in the lattice of von Neumann subalgebras of a fixed ambient algebra. The functor $\mathcal{A}$ supplies a canonical ambient algebra $\mathcal{A}(T)$ for each region $T$, so it is natural to package ``gluing inside $\mathcal{A}(T)$'' as a statement in the poset of subalgebras.

Fix $T\in \mathcal{R}(M)$. Let $\mathbf{vSub}\bigl(\mathcal{A}(T)\bigr)$ denote the thin category (poset) of von Neumann subalgebras of $\mathcal{A}(T)$, with a unique morphism $B\to C$ precisely when $B\subseteq C$.

Let $\mathbf{Inc}(M)_{\le T}$ be the thin category whose objects are inclusions $i_{U,T}:U\hookrightarrow T$ and with a unique morphism $i_{U,T}\to i_{V,T}$ exactly when $U\subseteq V\subseteq T$. Define a functor
$$
\widetilde{\mathcal{A}}_T:\mathbf{Inc}(M)_{\le T}\longrightarrow \mathbf{vSub}\bigl(\mathcal{A}(T)\bigr),
\qquad
\widetilde{\mathcal{A}}_T(U):=\mathcal{A}(i_{U,T})\bigl(\mathcal{A}(U)\bigr)\subseteq \mathcal{A}(T).
$$

Assume HK5 (Additivity): for any admissible cover $W=\bigcup_\alpha W_\alpha$ with inclusions $i_\alpha:W_\alpha\hookrightarrow W$ one has
$$
\mathcal{A}(W)=\bigvee_\alpha \mathcal{A}(i_\alpha)\bigl(\mathcal{A}(W_\alpha)\bigr),
$$
where $\vee$ denotes the von Neumann algebra generated inside $\mathcal{A}(W)$.

\begin{lemma}[Binary unions become joins under HK5]\label{lem:union-join}
Suppose $U,V,U\cup V\in \mathcal{R}(M)$ and $U,V\subseteq T$. Then, inside $\mathcal{A}(T)$,
$$
\widetilde{\mathcal{A}}_T(U\cup V)=\widetilde{\mathcal{A}}_T(U)\,\vee\,\widetilde{\mathcal{A}}_T(V).
$$
\end{lemma}

\begin{proof}
Consider the factorisations of inclusions
$$
U \xrightarrow{i_{U,U\cup V}} U\cup V \xrightarrow{i_{U\cup V,T}} T,
\qquad
V \xrightarrow{i_{V,U\cup V}} U\cup V \xrightarrow{i_{U\cup V,T}} T.
$$
Functoriality of $\mathcal{A}$ gives
$$
\mathcal{A}(i_{U,T})=\mathcal{A}(i_{U\cup V,T})\circ \mathcal{A}(i_{U,U\cup V}),
\qquad
\mathcal{A}(i_{V,T})=\mathcal{A}(i_{U\cup V,T})\circ \mathcal{A}(i_{V,U\cup V}).
$$
Apply HK5 to the cover $U\cup V = U\cup V$ (two-set cover by $U$ and $V$). Then $\mathcal{A}(U\cup V)$ is generated by the images of $\mathcal{A}(U)$ and $\mathcal{A}(V)$ in $\mathcal{A}(U\cup V)$ under $\mathcal{A}(i_{U,U\cup V})$ and $\mathcal{A}(i_{V,U\cup V})$. Applying $\mathcal{A}(i_{U\cup V,T})$ and using the two displayed factorisations yields
$$
\mathcal{A}(i_{U\cup V,T})\bigl(\mathcal{A}(U\cup V)\bigr)
=
\mathcal{A}(i_{U,T})\bigl(\mathcal{A}(U)\bigr)\,\vee\,\mathcal{A}(i_{V,T})\bigl(\mathcal{A}(V)\bigr)
\subseteq \mathcal{A}(T).
$$
By definition of $\widetilde{\mathcal{A}}_T$, this is exactly the claimed identity.
\end{proof}

\textbf{Inclusion and truncation.}
A standard move in AQFT is to enlarge the region class so that time-slice is ``local'' (or easier to impose), and then restrict back to a preferred region class. We record a clean formal package for this.

Fix:
\begin{enumerate}[label=\alph*)]
\item a region class $\mathcal{R}(M)$ and a larger region class $\mathcal{R}_{\mathrm{cau}}(M)$;
\item a strict inclusion functor of inclusion categories
$$
\iota:\mathbf{Inc}(M)\hookrightarrow \mathbf{Inc}_{\mathrm{cau}}(M);
$$
\item a functor (a ``truncation'')
$$
\tau:\mathbf{Inc}_{\mathrm{cau}}(M)\longrightarrow \mathbf{Inc}(M)
\quad\text{with}\quad
\tau\circ \iota=\mathrm{id}_{\mathbf{Inc}(M)};
$$
\item for each $U\in \mathcal{R}_{\mathrm{cau}}(M)$, a specified inclusion
$$
c_U:\tau(U)\hookrightarrow U
$$
which lies in the chosen class of Cauchy inclusions used in HK4.
\end{enumerate}

\textit{Remark.}
If $\iota$ and $\tau$ extend compatibly to the embedding categories (so that admissible embeddings are sent to admissible embeddings and commutative squares are preserved), then they induce double functors
$$
\iota_*:\mathbf{Mink}(M)\hookrightarrow \mathbf{Mink}_{\mathrm{cau}}(M),
\qquad
\tau_*:\mathbf{Mink}_{\mathrm{cau}}(M)\longrightarrow \mathbf{Mink}(M),
$$
with $\tau_*\circ \iota_*=\mathrm{id}$ strictly. We will only use the vertical part of this story in what follows.

\begin{lemma}[Adjunction criterion in thin categories]\label{lem:thin-adjunction}
Assume $\mathbf{Inc}(M)$ and $\mathbf{Inc}_{\mathrm{cau}}(M)$ are thin categories. Then $\tau\dashv \iota$ if and only if for all $U\in \mathcal{R}_{\mathrm{cau}}(M)$ and $V\in \mathcal{R}(M)$,
$$
\tau(U)\subseteq V \quad\Longleftrightarrow\quad U\subseteq \iota(V).
$$
\end{lemma}

\begin{proof}
In a thin category, $\mathrm{Hom}(X,Y)$ is either empty or a singleton, and it is nonempty exactly when $X\le Y$. An adjunction $\tau\dashv \iota$ is equivalent to natural bijections
$$
\mathrm{Hom}\bigl(\tau(U),V\bigr)\cong \mathrm{Hom}\bigl(U,\iota(V)\bigr).
$$
Such a bijection exists if and only if the two hom-sets are simultaneously empty or simultaneously nonempty, i.e.\ $\tau(U)\le V$ iff $U\le \iota(V)$.
\end{proof}

\textbf{Functoriality of $\mathcal{A}\mapsto \mathcal{F}_{\mathcal{A}}$ on morphisms.}
Let $\mathcal{A},\mathcal{B}:\mathbf{Emb}(M)\to \mathbf{vN}_{\mathrm{vert}}$ be AQFT inputs and let $\eta:\mathcal{A}\Rightarrow \mathcal{B}$ be a natural transformation. Define components
\begin{enumerate}[label=\alph*)]
\item on each region $U$,
$$
(\mathcal{F}_\eta)_U:=\eta_U:\mathcal{A}(U)\longrightarrow \mathcal{B}(U);
$$
\item on each admissible embedding $h:U\hookrightarrow V$ (a horizontal arrow of $\mathbf{Mink}(M)$),
$$
(\mathcal{F}_\eta)_h:=L^2(\eta_V):L^2\bigl(\mathcal{A}(V)\bigr)\longrightarrow L^2\bigl(\mathcal{B}(V)\bigr).
$$
\end{enumerate}

\begin{proposition}[Functoriality on 1-morphisms]\label{prop:functoriality-A-to-F}
The data $\mathcal{F}_\eta$ define a vertical transformation
$$
\mathcal{F}_\eta:\mathcal{F}_{\mathcal{A}}\Rightarrow \mathcal{F}_{\mathcal{B}}
$$
in the following explicit sense:
\begin{enumerate}[label=\alph*)]
\item for each horizontal arrow $h:U\hookrightarrow V$, the map $L^2(\eta_V)$ is an intertwiner
$$
(\mathcal{F}_\eta)_h:\mathcal{F}_{\mathcal{A}}(h)\longrightarrow \mathcal{F}_{\mathcal{B}}(h)
$$
whose left and right boundary maps are $(\eta_U,\eta_V)$; and
\item these components are natural with respect to every commuting square in $\mathbf{Mink}(M)$.
\end{enumerate}
Moreover, $\eta\mapsto \mathcal{F}_\eta$ respects identities and composition:
$$
\mathcal{F}_{\mathrm{id}_{\mathcal{A}}}=\mathrm{id}_{\mathcal{F}_{\mathcal{A}}},
\qquad
\mathcal{F}_{\theta\circ \eta}=\mathcal{F}_\theta\circ \mathcal{F}_\eta.
$$
\end{proposition}

\begin{proof}
Let $\eta:\mathcal{A}\Rightarrow \mathcal{B}$ be natural. Thus for each admissible embedding $h:U\hookrightarrow V$ in $\mathbf{Emb}(M)$ the naturality square commutes:
$$
\eta_V\circ \mathcal{A}(h)=\mathcal{B}(h)\circ \eta_U.
$$

Fix a horizontal arrow $h:U\hookrightarrow V$ in $\mathbf{Mink}(M)$. By definition,
$\mathcal{F}_{\mathcal{A}}(h)$ is the correspondence
$
\mathcal{A}(U)\rightsquigarrow \mathcal{A}(V)
$
with underlying Hilbert space $L^2\bigl(\mathcal{A}(V)\bigr)$, standard right action of $\mathcal{A}(V)$, and left action of $\mathcal{A}(U)$ restricted along $\mathcal{A}(h)$. Similarly, $\mathcal{F}_{\mathcal{B}}(h)$ has underlying Hilbert space $L^2\bigl(\mathcal{B}(V)\bigr)$ with left action restricted along $\mathcal{B}(h)$.

Define $(\mathcal{F}_\eta)_h:=L^2(\eta_V)$. This is defined because $\eta_V$ lies in $\mathbf{vN}_{\mathrm{vert}}$ and $L^2(-)$ is functorial on $\mathbf{vN}_{\mathrm{vert}}$ and intertwines the standard left and right actions.

\textit{Right action.}
For $b\in \mathcal{A}(V)$, functoriality of standard form gives
$$
L^2(\eta_V)\bigl(\xi\cdot b\bigr)=L^2(\eta_V)(\xi)\cdot \eta_V(b),
$$
equivalently the commutative diagram
% https://q.uiver.app/#q=WzAsNCxbMCwwLCJMXjIoXFxtYXRoY2Fse0F9KFYpKSJdLFsxLDAsIkxeMihcXG1hdGhjYWx7Qn0oVikpIl0sWzAsMSwiTF4yKFxtYXRoY2Fse0F9KFYpKSJdLFsxLDEsIkxeMihcXG1hdGhjYWx7Qn0oVikpIl0sWzAsMSwiTF4yKFxcZXRhX1YpIl0sWzIsMywiTF4yKFxcZXRhX1YpIiwyXSxbMCwyLCJcXHJob197XFxtYXRoY2Fse0F9KFYpfShiKSIsMl0sWzEsMywiXFxyaG9fe1xcbWF0aGNhbHtCfShWKX0oXFxldGFfVihiKSkiXV0=
\[\begin{tikzcd}
	{L^2(\mathcal{A}(V))} & {L^2(\mathcal{B}(V))} \\
	{L^2(\mathcal{A}(V))} & {L^2(\mathcal{B}(V))}
	\arrow["{L^2(\eta_V)}", from=1-1, to=1-2]
	\arrow["{\rho_{\mathcal{A}(V)}(b)}"', from=1-1, to=2-1]
	\arrow["{\rho_{\mathcal{B}(V)}(\eta_V(b))}", from=1-2, to=2-2]
	\arrow["{L^2(\eta_V)}"', from=2-1, to=2-2]
\end{tikzcd}\]
which is exactly the intertwiner condition for the right boundary map $\eta_V$.

\textit{Restricted left action.}
For $a\in \mathcal{A}(U)$, the left action on $\mathcal{F}_{\mathcal{A}}(h)$ is by
$\lambda_{\mathcal{A}(V)}\bigl(\mathcal{A}(h)(a)\bigr)$.
Using the standard-form left intertwining identity with $x=\mathcal{A}(h)(a)\in \mathcal{A}(V)$ gives
$$
L^2(\eta_V)\,\lambda_{\mathcal{A}(V)}\bigl(\mathcal{A}(h)(a)\bigr)
=
\lambda_{\mathcal{B}(V)}\bigl(\eta_V(\mathcal{A}(h)(a))\bigr)\,L^2(\eta_V).
$$
By naturality, $\eta_V(\mathcal{A}(h)(a))=\mathcal{B}(h)(\eta_U(a))$, so the right-hand side is precisely
$\lambda_{\mathcal{B}(V)}\bigl(\mathcal{B}(h)(\eta_U(a))\bigr)\,L^2(\eta_V)$,
which is the restricted left action corresponding to the left boundary map $\eta_U$. Equivalently, the diagram
% https://q.uiver.app/#q=WzAsNCxbMCwwLCJMXjIoXFxtYXRoY2Fse0F9KFYpKSJdLFsxLDAsIkxeMihcXG1hdGhjYWx7Qn0oVikpIl0sWzAsMSwiTF4yKFxtYXRoY2Fse0F9KFYpKSJdLFsxLDEsIkxeMihcXG1hdGhjYWx7Qn0oVikpIl0sWzAsMSwiTF4yKFxcZXRhX1YpIl0sWzIsMywiTF4yKFxcZXRhX1YpIiwyXSxbMCwyLCJcXGxhbWJkYV97XFxtYXRoY2Fse0F9KFYpfShcXG1hdGhjYWx7QX0oaCkoYSkpIiwyXSxbMSwzLCJcXGxhbWJkYV97XFxtYXRoY2Fse0J9KFYpfShcXG1hdGhjYWx7Qn0oaCkoXFxldGFfVShhKSkpIl1d
\[\begin{tikzcd}
	{L^2(\mathcal{A}(V))} & {L^2(\mathcal{B}(V))} \\
	{L^2(\mathcal{A}(V))} & {L^2(\mathcal{B}(V))}
	\arrow["{L^2(\eta_V)}", from=1-1, to=1-2]
	\arrow["{\lambda_{\mathcal{A}(V)}(\mathcal{A}(h)(a))}"', from=1-1, to=2-1]
	\arrow["{\lambda_{\mathcal{B}(V)}(\mathcal{B}(h)(\eta_U(a)))}", from=1-2, to=2-2]
	\arrow["{L^2(\eta_V)}"', from=2-1, to=2-2]
\end{tikzcd}\]
commutes. This is the intertwiner condition for the restricted left actions with left boundary map $\eta_U$.

So $(\mathcal{F}_\eta)_h$ is a well-typed square in $\mathbf{vNA}$ from $\mathcal{F}_{\mathcal{A}}(h)$ to $\mathcal{F}_{\mathcal{B}}(h)$ with vertical boundary $(\eta_U,\eta_V)$.

\textit{Naturality with respect to spacetime squares.}
Let a square in $\mathbf{Mink}(M)$ be given by a commuting diagram
$$
\begin{tikzcd}
	U & V \\
	{U'} & {V'}
	\arrow["h", from=1-1, to=1-2]
	\arrow["i"', from=1-1, to=2-1]
	\arrow["j", from=1-2, to=2-2]
	\arrow["{h'}"', from=2-1, to=2-2]
\end{tikzcd}
\qquad (j\circ h = h'\circ i).
$$
By definition, $\mathcal{F}_{\mathcal{A}}$ sends this square to $L^2\bigl(\mathcal{A}(j)\bigr)$ and $\mathcal{F}_{\mathcal{B}}$ sends it to $L^2\bigl(\mathcal{B}(j)\bigr)$. We must check commutativity of
% https://q.uiver.app/#q=WzAsNCxbMCwwLCJMXjIoXFxtYXRoY2Fse0F9KFYpKSJdLFsxLDAsIkxeMihcXG1hdGhjYWx7QX0oVicpKSJdLFswLDEsIkxeMihcXG1hdGhjYWx7Qn0oVikpIl0sWzEsMSwiTF4yKFxtYXRoY2Fse0J9KFYnKSkiXSxbMCwxLCJMXjIoXFxtYXRoY2Fse0F9KGopKSJdLFsyLDMsIkxeMihcXG1hdGhjYWx7Qn0oailpIiwyXSxbMCwyLCJMXjIoXFxldGFfVikiLDJdLFsxLDMsIkxeMihcXGV0YV97Vid9KSJdXQ==
\[\begin{tikzcd}
	{L^2(\mathcal{A}(V))} & {L^2(\mathcal{A}(V'))} \\
	{L^2(\mathcal{B}(V))} & {L^2(\mathcal{B}(V'))}
	\arrow["{L^2(\mathcal{A}(j))}", from=1-1, to=1-2]
	\arrow["{L^2(\eta_V)}"', from=1-1, to=2-1]
	\arrow["{L^2(\eta_{V'})}", from=1-2, to=2-2]
	\arrow["{L^2(\mathcal{B}(j))}"', from=2-1, to=2-2]
\end{tikzcd}\]
This commutes because $\eta$ is natural at $j:V\hookrightarrow V'$:
$$
\eta_{V'}\circ \mathcal{A}(j)=\mathcal{B}(j)\circ \eta_V,
$$
and applying the functor $L^2(-)$ gives
$$
L^2(\eta_{V'})\circ L^2\bigl(\mathcal{A}(j)\bigr)
=
L^2\bigl(\mathcal{B}(j)\bigr)\circ L^2(\eta_V).
$$
This is exactly the required square-naturality.

\textit{Identities and composition.}
If $\eta=\mathrm{id}_{\mathcal{A}}$, then $(\mathcal{F}_\eta)_U=\mathrm{id}_{\mathcal{A}(U)}$ and
$(\mathcal{F}_\eta)_h=L^2\bigl(\mathrm{id}_{\mathcal{A}(V)}\bigr)=\mathrm{id}_{L^2(\mathcal{A}(V))}$, so
$\mathcal{F}_{\mathrm{id}_{\mathcal{A}}}=\mathrm{id}_{\mathcal{F}_{\mathcal{A}}}$.
If $\mathcal{A}\xRightarrow{\eta}\mathcal{B}\xRightarrow{\theta}\mathcal{C}$, then for each $h:U\hookrightarrow V$,
$$
(\mathcal{F}_{\theta\circ \eta})_h
=
L^2\bigl((\theta\circ \eta)_V\bigr)
=
L^2(\theta_V\circ \eta_V)
=
L^2(\theta_V)\circ L^2(\eta_V)
=
(\mathcal{F}_\theta)_h\circ (\mathcal{F}_\eta)_h,
$$
and similarly on objects. Hence $\mathcal{F}_{\theta\circ \eta}=\mathcal{F}_\theta\circ \mathcal{F}_\eta$.

This proves all claims.
\end{proof}

We now move on to examples illustrating how these structural statements show up in concrete AQFT inputs.

\section{Examples}

Throughout this section, the domain double category is $\mathbf{Mink}(M)$ (or $\mathbf{Mink}(S^{1})$ in the circle example). Its objects are regions $U$, its vertical arrows are inclusions $i:U\hookrightarrow U'$, and its horizontal arrows are admissible embeddings $h:U\hookrightarrow V$ (as specified in each example). A square is a commuting boundary diagram:

% https://q.uiver.app/#q=WzAsNCxbMCwxLCJVJyJdLFsxLDAsIlYiXSxbMSwxLCJWLyJdLFswLDAsIlUiXSxbMCwyLCJoJyIsMl0sWzEsMiwiaiJdLFszLDEsImgiXSxbMywwLCJpIiwyXV0=
$$
\begin{tikzcd}
	U & V \\
	{U'} & {V'}
	\arrow["h", from=1-1, to=1-2]
	\arrow["i"', from=1-1, to=2-1]
	\arrow["j", from=1-2, to=2-2]
	\arrow["{h'}"', from=2-1, to=2-2]
\end{tikzcd}
\quad \text{meaning } j\circ h = h'\circ i \text{ as maps.}
$$

The input lands in $\mathbf{vN}_{\mathrm{vert}}$: its objects are von Neumann algebras and its morphisms are normal unital $*$-homomorphisms (in the class for which $\phi\mapsto L^{2}(\phi)$ is functorial). Given an input $\mathcal{A}$, the associated double functor $\mathcal{F}_{\mathcal{A}}$ is obtained by the fixed standard-form/correspondence construction, using the functoriality of $\phi\mapsto L^{2}(\phi)$ on $\mathbf{vN}_{\mathrm{vert}}$.

We assume basic familiarity with operator-algebraic AQFT (von Neumann algebras and conformal nets, including $\mathrm{Diff}_{+}(S^{1})$-covariance) and with the Weyl/CCR construction of free fields.

\textbf{Diffeomorphism-covariant conformal nets on $S^1$.} Let $(\mathcal{B},U,\Omega)$ be a $\mathrm{Diff}_{+}(S^{1})$-covariant conformal net on $S^{1}$ in the standard sense: for each nonempty, nondense open arc $I\subset S^{1}$ there is a von Neumann algebra
$
\mathcal{B}(I)\subset B(\mathcal{H}),
$
isotone in $I$, local for disjoint arcs, and covariant under $\mathrm{Diff}_{+}(S^{1})$ via a (possibly projective) unitary representation $U$. We use only that the induced maps
$
\operatorname{Ad}_{U(F)}:x\longmapsto U(F)\,x\,U(F)^{*}
$
define an honest action on the algebras (projective phases cancel under $\operatorname{Ad}$).

Define the quasilocal algebra (in this represented form) by
$$
\mathcal{B}(S^{1})
:=
\Bigl(\bigcup_{I\subset S^{1}\ \mathrm{arc}}\mathcal{B}(I)\Bigr)^{\prime\prime}
\subset
B(\mathcal{H}).
$$
Thus $\mathcal{B}(S^{1})$ is the von Neumann algebra generated by the local arc algebras.

We make admissible embeddings part of the data by requiring a chosen global extension.

\begin{itemize}
\item Objects. Nonempty, nondense open arcs $I\subset S^{1}$ together with the ambient object $S^{1}$.

\item Vertical arrows. Inclusions of arcs $i:I\hookrightarrow I'$ and the structure inclusions $I\hookrightarrow S^{1}$.

\item Horizontal arrows. A horizontal arrow $h:I\to J$ is a pair
$$
h=(F,I),
$$
where $F\in \mathrm{Diff}_{+}(S^{1})$ satisfies $F(I)\subseteq J$. Its underlying smooth embedding is the restriction $F|_{I}:I\to J$.

\item Horizontal composition. If $h_{1}=(F_{1},I_{1}):I_{1}\to I_{2}$ and $h_{2}=(F_{2},I_{2}):I_{2}\to I_{3}$, define
$$
h_{2}\circ h_{1}:=(F_{2}F_{1},I_{1}),
$$
whose underlying map is $(F_{2}F_{1})|_{I_{1}}$. Associativity is inherited from composition in $\mathrm{Diff}_{+}(S^{1})$.

\item Squares. A square is a commuting square in $\mathbf{Emb}(S^{1})$, i.e. a boundary diagram with vertical inclusions $i:I\hookrightarrow I'$ and $j:J\hookrightarrow J'$ and horizontals $h=(F,I)$ and $h'=(F',I')$ such that
$$
j\circ h = h'\circ i
$$
as morphisms in $\mathbf{Emb}(S^{1})$. Concretely, this forces $F'=F$ and the underlying embeddings commute:
$$
j\circ (F|_{I})=(F|_{I'})\circ i.
$$
\end{itemize}

Define
$$
\mathcal{A}:\mathbf{Emb}(S^{1})\longrightarrow \mathbf{vN}_{\mathrm{vert}}
$$
as follows:

\begin{itemize}
\item Objects. For an arc $I\subset S^{1}$ set $\mathcal{A}(I):=\mathcal{B}(I)$, and set
$$
\mathcal{A}(S^{1}):=\mathcal{B}(S^{1}).
$$

\item Vertical arrows. For an inclusion $i:I\hookrightarrow I'$ of arcs, define $\mathcal{A}(i)$ to be the inclusion
$$
\mathcal{A}(i):\mathcal{B}(I)\hookrightarrow \mathcal{B}(I').
$$
This is a normal injective unital $*$-homomorphism.

\item Horizontal arrows. For a horizontal arrow $h=(F,I):I\to J$ (so $F\in \mathrm{Diff}_{+}(S^{1})$ and $F(I)\subseteq J$), define
$$
\mathcal{A}(h)\;:=\;\mathrm{Ad}_{U(F)}\big|_{\mathcal{B}(I)}:\mathcal{B}(I)\longrightarrow \mathcal{B}(F(I))\subseteq \mathcal{B}(J).
$$
This is a normal unital $*$-homomorphism, since $\mathrm{Ad}_{U(F)}$ is a $*$-automorphism of $B(\mathcal{H})$ and covariance gives
$\mathrm{Ad}_{U(F)}\bigl(\mathcal{B}(I)\bigr)=\mathcal{B}(F(I))$.
\end{itemize}

Functoriality on horizontal arrows holds at the algebra level: for $F_{1},F_{2}\in \mathrm{Diff}_{+}(S^{1})$,
$$
\mathrm{Ad}_{U(F_{2})}\circ \mathrm{Ad}_{U(F_{1})}
=\mathrm{Ad}_{U(F_{2})U(F_{1})}
=\mathrm{Ad}_{U(F_{2}F_{1})},
$$
where the last equality uses that any projective scalar in $U(F_{2})U(F_{1})$ acts trivially under $\mathrm{Ad}$.

We now verify each of the Haag--Kastler axioms for $\mathcal{A}$.

\begin{itemize}
\item \textbf{HK1 (Isotony).}
If $i:I\hookrightarrow I'$ is a vertical inclusion of arcs, then by construction
$$
\mathcal{A}(i):\mathcal{A}(I)=\mathcal{B}(I)\longrightarrow \mathcal{A}(I')=\mathcal{B}(I')
$$
is literally the inclusion $\mathcal{B}(I)\subseteq \mathcal{B}(I')$.

\item \textbf{HK2 (Locality).}
Let $I,J\subset S^{1}$ be disjoint arcs. Locality of the conformal net gives
$$
[\mathcal{B}(I),\mathcal{B}(J)]=0 \qquad \text{inside } B(\mathcal{H}).
$$
Since $\mathcal{A}(S^{1})=\mathcal{B}(S^{1})\subseteq B(\mathcal{H})$ contains both $\mathcal{B}(I)$ and $\mathcal{B}(J)$ as von Neumann subalgebras, the same commutator identity holds inside the ambient algebra $\mathcal{A}(S^{1})$.

\item \textbf{HK3 (Covariance).}
In this setup, \textbf{HK3} is exactly the statement that $\mathcal{A}$ is functorial on the embedding category.
For horizontal arrows $h_{1}=(F_{1},I_{1}):I_{1}\to I_{2}$ and $h_{2}=(F_{2},I_{2}):I_{2}\to I_{3}$,
$$
\mathcal{A}(h_{2})\circ \mathcal{A}(h_{1})
=\mathrm{Ad}_{U(F_{2})}\circ \mathrm{Ad}_{U(F_{1})}
=\mathrm{Ad}_{U(F_{2}F_{1})}
=\mathcal{A}\bigl((F_{2}F_{1},I_{1})\bigr),
$$
using $\mathrm{Ad}_{U(F_{2})}\circ \mathrm{Ad}_{U(F_{1})}=\mathrm{Ad}_{U(F_{2})U(F_{1})}$ and that any projective scalar in
$U(F_{2})U(F_{1})$ disappears under $\mathrm{Ad}$.
Vertical functoriality is immediate because vertical arrows are inclusions and $\mathcal{A}$ sends them to inclusions.

\item \textbf{HK4 (Time-slice).}
In the circle setup used here, we take the designated class of Cauchy inclusions to be empty, hence \textbf{HK4} is vacuous for this example.

\item \textbf{HK5 (Additivity).}
We use the standard additivity axiom of the chosen conformal net: if an arc $I$ is covered by subarcs $\{I_{\alpha}\}$, then
$$
\mathcal{B}(I)=\Bigl(\bigcup_{\alpha}\mathcal{B}(I_{\alpha})\Bigr)''.
$$
Since $\mathcal{A}(I)=\mathcal{B}(I)$ on objects and $\mathcal{A}$ sends inclusions of subarcs to the corresponding inclusions of algebras, this is exactly \textbf{HK5} for the extracted net $\mathcal{A}\!\mid_{\mathbf{Inc}(S^{1})}$.
\end{itemize}

Now we verify square compatibility. Let a square in $\mathbf{Mink}(S^{1})$ be specified by vertical inclusions $i:I\hookrightarrow I'$ and $j:J\hookrightarrow J'$, and by horizontal arrows
$$
h=(F,I):I\to J,
\quad
h'=(F',I'):I'\to J',
$$
such that $j\circ h = h'\circ i$ in $\mathbf{Emb}(S^{1})$. In particular, $F'=F$, and the underlying embeddings commute:
$$
j\circ (F|_{I}) \;=\; (F|_{I'})\circ i.
$$
Then the two composites
$$
\mathcal{A}(j)\circ \mathcal{A}(h),
\quad
\mathcal{A}(h')\circ \mathcal{A}(i)
$$
agree as normal unital $*$-homomorphisms $\mathcal{B}(I)\to \mathcal{B}(J')$.
Indeed, $\mathcal{A}(i)$ and $\mathcal{A}(j)$ are the inclusion maps by \textbf{HK1}, while
$$
\mathcal{A}(h)=\mathrm{Ad}_{U(F)}\big|_{\mathcal{B}(I)}
\quad\text{and}\quad
\mathcal{A}(h')=\mathrm{Ad}_{U(F)}\big|_{\mathcal{B}(I')}.
$$
Thus for $x\in \mathcal{B}(I)$ both composites send $x$ to $\mathrm{Ad}_{U(F)}(x)$, viewed inside $\mathcal{B}(J')$ by isotony. This is exactly the square-compatibility condition required for the input functor $\mathcal{A}$.

\textbf{The massive Klein--Gordon net on $\mathbb{R}^{1,3}$.} In this example, let $M$ be Minkowski spacetime and let $R(M)$ be the class of double cones (diamonds), together with $M$ itself as an ambient object. Horizontal morphisms are taken to be restrictions of proper orthochronous Poincar\'e isometries (i.e.\ time-orientation preserving), so that the retarded/advanced structures are preserved.

Let $P_M := \Box_M + m^2$ act on $C^\infty(M,\mathbb{R})$. This operator is formally self-adjoint with respect to the Minkowski volume form. Since $M$ is globally hyperbolic (in particular, Minkowski space), there exist unique advanced and retarded Green operators
$$
E_M^\pm : C_c^\infty(M) \longrightarrow C^\infty(M)
$$
with the usual support and inverse properties. Define the causal propagator by
$$
E_M := E_M^- - E_M^+ .
$$
In what follows, we use only these standard existence, uniqueness, and support properties \cite{baer2008waveequationslorentzianmanifolds}.

For each diamond $T\subset M$, define the real vector space
$$
V(T)\;:=\;C_c^\infty(T,\mathbb{R})\,/\,P_T C_c^\infty(T,\mathbb{R}),
$$
where $P_T$ denotes the restriction of $P_M$ to test functions supported in $T$.

If $i:T\hookrightarrow T'$ is a vertical inclusion, define
$$
V(i):V(T)\longrightarrow V(T'),\quad V(i)[f]\;:=\;[f_{\mathrm{ext},T'}],
$$
where $f_{\mathrm{ext},T'}\in C_c^\infty(T')$ is the extension of $f$ by zero outside $T$.

\begin{lemma}[Extension commutes with $P$]\label{lem:ext-commutes-P}
For every $u\in C_c^\infty(T)$ one has
$$
(P_Tu)_{\mathrm{ext},T'} \;=\; P_{T'}(u_{\mathrm{ext},T'}).
$$
\end{lemma}

\begin{proof}
Since $\operatorname{supp}(u)\Subset T$ and $T\subset T'$ is open, the extension $u_{\mathrm{ext},T'}$ is smooth on $T'$ and agrees with $u$ on a neighborhood of $\operatorname{supp}(u)$. As $P$ is a differential operator, it is local; hence $P_{T'}(u_{\mathrm{ext},T'})$ agrees with $P_Tu$ on a neighborhood of $\operatorname{supp}(u)$ and vanishes outside $T$. Therefore $P_{T'}(u_{\mathrm{ext},T'})$ is precisely the extension by zero of $P_Tu$, i.e.\ $(P_Tu)_{\mathrm{ext},T'}=P_{T'}(u_{\mathrm{ext},T'})$.
\end{proof}

It follows that $V(i)$ is well-defined on the quotient: if $f=P_Tu$, then
$$
f_{\mathrm{ext},T'} \;=\; (P_Tu)_{\mathrm{ext},T'} \;=\; P_{T'}(u_{\mathrm{ext},T'}) \;\in\; P_{T'}C_c^\infty(T').
$$

For each diamond $T\subset M$, define a bilinear form
$$
\sigma_T([f],[g])\;:=\;\int_M f_{\mathrm{ext},M}\,E_M\!\bigl(g_{\mathrm{ext},M}\bigr)\,d\mathrm{vol}_M.
$$
(We always extend to the ambient spacetime $M$, so the formula is evaluated in a fixed domain.)

\begin{lemma}[Well-definedness on the quotient]\label{lem:sigma-quotient}
The bilinear form $\sigma_T$ descends to a well-defined antisymmetric bilinear form on
$$
V(T)\;=\;C_c^\infty(T)\,/\,P_T C_c^\infty(T).
$$
\end{lemma}

\begin{proof}
Independence of representatives (first slot).
Suppose $f=P_Tu$ for some $u\in C_c^\infty(T)$. By Lemma~\ref{lem:ext-commutes-P} with $T'=M$,
$$
f_{\mathrm{ext},M}=(P_Tu)_{\mathrm{ext},M}=P_M(u_{\mathrm{ext},M}).
$$
Hence for any $g\in C_c^\infty(T)$,
$$
\sigma_T([P_Tu],[g])
=\int_M P_M(u_{\mathrm{ext},M})\,E_M(g_{\mathrm{ext},M})\,d\mathrm{vol}_M.
$$
Using formal self-adjointness of $P_M$ on compactly supported test functions, we may integrate by parts:
$$
\int_M P_M(u_{\mathrm{ext},M})\,v\,d\mathrm{vol}_M
=\int_M u_{\mathrm{ext},M}\,P_M(v)\,d\mathrm{vol}_M
\quad (v\in C^\infty(M)).
$$
Taking $v=E_M(g_{\mathrm{ext},M})$ and using $P_M E_M=0$, we obtain
$\sigma_T([P_Tu],[g])=0$.

Independence of representatives (second slot).
The same argument shows $\sigma_T([f],[P_Tv])=0$ for $v\in C_c^\infty(T)$, using $E_M P_M=0$.

Antisymmetry.
The causal propagator satisfies $E_M^*=-E_M$, equivalently
$$
\int_M f\,E_M(g)\,d\mathrm{vol}_M \;=\;-\int_M g\,E_M(f)\,d\mathrm{vol}_M
\quad (f,g\in C_c^\infty(M)).
$$
Applying this to $f_{\mathrm{ext},M},g_{\mathrm{ext},M}$ gives
$\sigma_T([f],[g])=-\sigma_T([g],[f])$.

Therefore $\sigma_T$ is well-defined on $V(T)$ and antisymmetric.
\end{proof}

\textit{Remark.} Nondegeneracy is standard and is not needed beyond the Weyl/CCR construction. If one wishes to avoid any degeneracy issue at the level of the CCR presentation, one may replace $V(T)$ by the symplectic quotient $V(T)/\mathrm{rad}(\sigma_T)$ throughout; we suppress this in the notation.

Let $h:T\to T'$ be a horizontal arrow given by restricting a proper orthochronous Poincar\'e isometry
$\kappa$ with $\kappa(T)\subset T'$, i.e.
$$
h=\kappa|_{T}:T\longrightarrow T'.
$$
Define the pushforward on test functions by
$$
(\kappa_* f)(x)\;:=\;f(\kappa^{-1}x),
$$
so that
$$
\kappa_*:C_c^\infty(T)\longrightarrow C_c^\infty(\kappa(T)).
$$
Define
$$
V(h):V(T)\longrightarrow V(T'),
\quad
V(h)[f]\;:=\;\bigl[(\kappa_* f)_{\mathrm{ext},T'}\bigr].
$$
This is well-defined on the quotient: if $f=P_Tu$, then $\kappa_* f=P_{\kappa(T)}(\kappa_*u)$ by
isometry-invariance of $P$, hence $(\kappa_* f)_{\mathrm{ext},T'}\in P_{T'}C_c^\infty(T')$.

\begin{lemma}[Pushforward commutes with extension]\label{lem:kappa-ext}
For any diffeomorphism $\kappa$ and $f\in C_c^\infty(T)$,
$$
\kappa_*(f_{\mathrm{ext},M}) \;=\; (\kappa_* f)_{\mathrm{ext},M}.
$$
\end{lemma}

\begin{proof}
Both sides are the test function on $M$ given by $x\mapsto f(\kappa^{-1}x)$ on $\kappa(T)$ and $0$
on $M\setminus \kappa(T)$.
\end{proof}

\begin{lemma}[Symplecticity]\label{lem:symplecticity}
For every horizontal arrow $h=\kappa|_{T}:T\to T'$,
$$
\sigma_{T'}\bigl(V(h)[f],\,V(h)[g]\bigr)\;=\;\sigma_T([f],[g])
\qquad
\text{for all }[f],[g]\in V(T).
$$
\end{lemma}

\begin{proof}
By definition of $V(h)$ and $\sigma_{T'}$,
$$
\sigma_{T'}\bigl(V(h)[f],V(h)[g]\bigr)
=\int_M (\kappa_*f)_{\mathrm{ext},M}\,E_M\bigl((\kappa_*g)_{\mathrm{ext},M}\bigr)\,d\mathrm{vol}_M.
$$
Using Lemma~\ref{lem:kappa-ext} (pushforward commutes with extension),
$$
=\int_M \kappa_*(f_{\mathrm{ext},M})\,E_M\bigl(\kappa_*(g_{\mathrm{ext},M})\bigr)\,d\mathrm{vol}_M.
$$
Since $\kappa$ is a time-orientation preserving isometry, uniqueness of advanced/retarded Green operators implies
$E_M^\pm\circ \kappa_*=\kappa_*\circ E_M^\pm$, hence $E_M\circ \kappa_*=\kappa_*\circ E_M$. Therefore
$$
=\int_M \kappa_*(f_{\mathrm{ext},M})\,\kappa_*\!\bigl(E_M(g_{\mathrm{ext},M})\bigr)\,d\mathrm{vol}_M.
$$
Since $\kappa$ preserves $d\mathrm{vol}_M$, change of variables yields
$$
=\int_M f_{\mathrm{ext},M}\,E_M(g_{\mathrm{ext},M})\,d\mathrm{vol}_M
=\sigma_T([f],[g]).
$$
\end{proof}

\begin{lemma}[Functoriality of $V(h)$]\label{lem:V-functoriality}
If $h_1=\kappa_1|_{T_1}:T_1\to T_2$ and $h_2=\kappa_2|_{T_2}:T_2\to T_3$, then
$$
V(h_2\circ h_1)\;=\;V(h_2)\circ V(h_1).
$$
\end{lemma}

\begin{proof}
Pushforwards compose: $(\kappa_2\kappa_1)_*=\kappa_{2*}\circ \kappa_{1*}$. Hence for $f\in C_c^\infty(T_1)$,
$$
V(h_2)\bigl(V(h_1)[f]\bigr)
=V(h_2)\bigl[(\kappa_{1*}f)_{\mathrm{ext},T_2}\bigr]
=\bigl[(\kappa_{2*}(\kappa_{1*}f)_{\mathrm{ext},T_2})_{\mathrm{ext},T_3}\bigr].
$$
But $(\kappa_{1*}f)_{\mathrm{ext},T_2}$ agrees with $\kappa_{1*}f$ on $\kappa_1(T_1)$ and is $0$ on $T_2\setminus \kappa_1(T_1)$, hence $\kappa_{2*}(\kappa_{1*}f)_{\mathrm{ext},T_2}$ agrees with $(\kappa_2\kappa_1)_*f$ on $\kappa_2\kappa_1(T_1)$ and is $0$ off $\kappa_2\kappa_1(T_1)$ inside $\kappa_2(T_2)$. Since $\kappa_2\kappa_1(T_1)\subset T_3$, extending by zero to $T_3$ yields
$$
\bigl[(\kappa_{2*}(\kappa_{1*}f)_{\mathrm{ext},T_2})_{\mathrm{ext},T_3}\bigr]
=\bigl[((\kappa_2\kappa_1)_*f)_{\mathrm{ext},T_3}\bigr]
=V(h_2\circ h_1)[f],
$$
so the equality holds on representatives and therefore on $V(T_1)$.
\end{proof}

For each diamond $T$, let $W(T)$ be the Weyl $C^*$-algebra generated by unitaries
$W_T(v)$ for $v\in V(T)$, subject to the CCR relations
$$
W_T(v)\,W_T(w)\;=\;e^{-\frac{i}{2}\sigma_T(v,w)}\,W_T(v+w),
\qquad
W_T(v)^*\;=\;W_T(-v).
$$
By Lemma~\ref{lem:symplecticity}, every horizontal map $V(h):V(T)\to V(T')$ induces a
$*$-homomorphism $\alpha_h:W(T)\to W(T')$ uniquely determined on generators by
$$
\alpha_h\bigl(W_T(v)\bigr)\;=\;W_{T'}\bigl(V(h)v\bigr).
$$
Likewise, each vertical inclusion $i:T\hookrightarrow T'$ induces $\alpha_i:W(T)\to W(T')$
via $V(i)$.

Let $W(M)$ denote the Weyl algebra of the symplectic space
$V(M)=C_c^\infty(M)/P_M C_c^\infty(M)$ with form $\sigma_M$.
For each $T$, extension-by-zero defines a linear map $V(T)\to V(M)$, hence a canonical
$*$-homomorphism
$$
\iota_T:W(T)\longrightarrow W(M).
$$
Set
$$
W_{\mathrm{ql}}\;:=\;\varinjlim_{\,T\in R(M)} W(T),
$$
and write $\jmath_T:W(T)\to W_{\mathrm{ql}}$ for the canonical structure maps to the inductive limit. Let $\pi_{\mathrm{univ}}:W_{\mathrm{ql}}\to B(H_{\mathrm{univ}})$ be the universal representation. Define, for each $T$,
$$
A(T)\;:=\;\pi_{\mathrm{univ}}\!\bigl(\jmath_T(W(T))\bigr)''\;\subset\;B(H_{\mathrm{univ}}).
$$
Since $\pi_{\mathrm{univ}}$ is faithful, the structure maps $W(T)\to W(T')$ in the inductive system
are realized as injective maps after applying $\pi_{\mathrm{univ}}(\,\cdot\,)''$, and for $T\subset T'$ we have the literal inclusion $A(T)\subset A(T')$ inside $B(H_{\mathrm{univ}})$.

Now, we define what $\mathcal{A}:\textbf{Emb}(M)\to \textbf{vN}_{\mathrm{vert}}$ is, and verify the Haag--Kastler axioms for them.

We define an AQFT input $A:\mathbf{Emb}(M)\to \mathbf{vN}_{\mathrm{vert}}$ as follows.
\begin{itemize}
\item Objects.
For each diamond $T\subset M$ set $A(T)$ as in the preceding construction, and include the ambient object
$M\mapsto A(M)$.

\item Vertical arrows.
For a vertical inclusion $i:T\hookrightarrow T'$, define $A(i):A(T)\to A(T')$ to be the normal inclusion
$A(T)\subset A(T')$.

\item Horizontal arrows.
Let $h:T\to T'$ be a horizontal arrow, given by the restriction $h=\kappa|_T$ of a proper orthochronous
Poincar\'e isometry $\kappa$ with $\kappa(T)\subset T'$. Let $$\alpha_h:W(T)\to W(T')$$ be the induced
$*$-homomorphism on Weyl algebras. Define $A(h):A(T)\to A(T')$ to be the unique normal $*$-homomorphism such that, for every
$v\in V(T)$,
$$
A(h)\Bigl(\pi_{\mathrm{univ}}\bigl(\jmath_T(W_T(v))\bigr)\Bigr)
\;:=\;
\pi_{\mathrm{univ}}\bigl(\jmath_{T'}(\alpha_h(W_T(v)))\bigr),
$$
and extend by normality to $A(T)=\pi_{\mathrm{univ}}(\jmath_T(W(T)))''$. This is well-defined since
$\alpha_h$ is a $*$-homomorphism and $\pi_{\mathrm{univ}}(\jmath_{T'}(W(T')))$ generates $A(T')$ as a
bicommutant.
\end{itemize}

The Haag--Kastler axioms for the extracted net are verified as follows:
\begin{itemize}
\item \textbf{HK1 (Isotony).}
For $i:T\hookrightarrow T'$ the map $A(i)$ is the literal inclusion $A(T)\subset A(T')$ by construction
(using the faithful universal representation).

\item \textbf{HK2 (Locality, via the ambient region $M$).}
Let $T_1,T_2$ be spacelike separated diamonds. For $f\in C_c^\infty(T_1)$ and $g\in C_c^\infty(T_2)$ set
$v_1:=[f_{\mathrm{ext},M}]$ and $v_2:=[g_{\mathrm{ext},M}]$ in $V(M)$. Then
$$\sigma_M(v_1,v_2)=\int_M f_{\mathrm{ext},M}\,E_M(g_{\mathrm{ext},M})\,d\mathrm{vol}_M=0$$ because
$$\mathrm{supp}\bigl(E_M(g_{\mathrm{ext},M})\bigr)\subset J(\mathrm{supp}(g))$$ is disjoint from
$\mathrm{supp}(f)$ under spacelike separation \cite{baer2008waveequationslorentzianmanifolds}. Hence the Weyl relations give
$$W_M(v_1)W_M(v_2)=W_M(v_2)W_M(v_1)$$ in $W(M)$.
Via the maps $\jmath_M:W(M)\to W_{\mathrm{ql}}$ and $\pi_{\mathrm{univ}}$, the same commutation holds inside $A(M)=\pi_{\mathrm{univ}}(\jmath_M(W(M)))''$. Since the inclusion $T_k\hookrightarrow M$ induces the $*$-homomorphism $\iota_{T_k}:W(T_k)\to W(M)$, the elements $W_M(v_k)$ are the images of $W_{T_k}([f])$ and $W_{T_k}([g])$ under $\iota_{T_k}$. Therefore the images of $W(T_1)$ and $W(T_2)$ commute in $W(M)$, hence their bicommutants commute in $A(M)$. This is exactly HK2 in the
formulation using orthogonality of inclusions into a common codomain.

\item \textbf{HK3 (Covariance as functoriality).}
The assignment $T\mapsto (V(T),\sigma_T)$ is functorial for inclusions and the chosen Poincar\'e
restrictions, and the Weyl functor sends symplectic maps to $*$-homomorphisms. Passing to
$\pi_{\mathrm{univ}}(\,\cdot\,)''$ yields a functor $A:\mathbf{Emb}(M)\to \mathbf{vN}_{\mathrm{vert}}$,
which is HK3 in the functorial sense used here.

\item \textbf{HK4 (Time-slice).}
In the double-cone region category with only inclusion verticals, we take the time-slice class to be
empty, $S=\varnothing$, so HK4 is vacuous in this setup.

\item \textbf{HK5 (Additivity).}
Let $T=\bigcup_{\lambda\in\Lambda}T_\lambda$ be an open cover by diamonds and let $f\in C_c^\infty(T)$.
Choose a finite subcover of $\mathrm{supp}(f)$ and a partition of unity $\{\chi_k\}$ subordinate to it,
so $f=\sum_k \chi_k f$ with each $\chi_k f\in C_c^\infty(T_{\lambda_k})$. Thus in $V(T)$,
$$[f]=\sum_k[\chi_k f]$$, and the Weyl relations express $W_T([f])$ as a phase times a product of the
generators $W_T([\chi_k f])$. Hence $W(T)$ is generated (as a $C^*$-algebra) by the images of the
subalgebras $W(T_{\lambda_k})$ under the inclusion maps. Applying $\pi_{\mathrm{univ}}(\,\cdot\,)''$
gives
$$
A(T)\;=\;\bigvee_{\lambda\in\Lambda} A(i_{T_\lambda,T})\bigl(A(T_\lambda)\bigr)
$$
inside $A(T)$, i.e.\ HK5.
\end{itemize}

Now, we check square-level bimodularity in $\mathcal{F}_{\mathcal{A}}.$

Let % https://q.uiver.app/#q=WzAsNCxbMCwxLCJUJyJdLFsxLDAsIlYiXSxbMSwxLCJWJyJdLFswLDAsIlQiXSxbMCwyLCJoJyIsMl0sWzEsMiwiaiJdLFszLDEsImgiXSxbMywwLCJpIiwyXV0=
$$
\begin{tikzcd}
	T & V \\
	{T'} & {V'}
	\arrow["h", from=1-1, to=1-2]
	\arrow["i"', from=1-1, to=2-1]
	\arrow["j", from=1-2, to=2-2]
	\arrow["{h'}"', from=2-1, to=2-2]
\end{tikzcd}
$$
be a square in $\mathbf{Mink}(M)$, so $j\circ h = h'\circ i$ as maps. Then, in the input, we have an
equality of normal $*$-homomorphisms
$$
A(j)\circ A(h) \;=\; A(h')\circ A(i).
$$

In the AQFT double functor $\mathcal{F}_\mathcal{A}$, the horizontal arrow $h:T\to V$ is sent to the correspondence
${}_{A(T)}L^2(A(V))_{A(V)}$, with left action $\lambda_V\circ A(h)$ and right action $\rho_V$ (standard form).
The square is sent to the standard-form map
$$
L^2(A(j)):\;L^2(A(V))\longrightarrow L^2(A(V')).
$$

Left bimodularity.
For $a\in A(T)$ and $\xi\in L^2(A(V))$,
\begin{align*}
L^2(A(j))\bigl(\lambda_V(A(h)(a))\,\xi\bigr)
&=\lambda_{V'}\bigl(A(j)(A(h)(a))\bigr)\,L^2(A(j))(\xi)\\
&=\lambda_{V'}\bigl(A(h')(A(i)(a))\bigr)\,L^2(A(j))(\xi),
\end{align*}
where the first equality is the standard-form intertwining identity for $L^2(-)$, and the second uses
$A(j)\circ A(h)=A(h')\circ A(i)$.

Right bimodularity.
For $b\in A(V)$,
$$
L^2(A(j))\bigl(\xi\,\rho_V(b)\bigr)
\;=\;
L^2(A(j))(\xi)\,\rho_{V'}\bigl(A(j)(b)\bigr),
$$
again by the standard-form intertwining identity.

These are exactly the bimodule-intertwiner identities required for the square assignment of $\mathcal{F}_\mathcal{A}$.

\textbf{Fixed-point subnets under finite groups.}
Let $(\mathcal{B},U,\Omega)$ be a strongly additive conformal net on $S^{1}$. Let $G$ be a finite group
acting by vacuum-preserving internal symmetries, i.e. for each $g\in G$ an automorphism
$
\beta_{g}\in \mathrm{Aut}\bigl(\mathcal{B}(S^{1})\bigr)
$
such that, for every arc $I\subset S^{1}$,
$$
\beta_{g}\bigl(\mathcal{B}(I)\bigr)=\mathcal{B}(I).
$$
Assume moreover that each $\beta_g$ is implemented on $\mathcal{H}$ by a unitary $V_g\in U(\mathcal{H})$
with $V_g\,\Omega=\Omega$ and
$$
\beta_g=\mathrm{Ad}_{V_g}\quad\text{on }\mathcal{B}(S^1)
$$
(and hence on each $\mathcal{B}(I)\subset B(\mathcal{H})$).%
\footnote{Equivalently, each $\beta_g$ restricts to an automorphism of every local algebra $\mathcal{B}(I)\subset B(\mathcal{H})$, and the implementing unitary fixes the vacuum vector.}
Assume moreover the following commutation hypothesis (used for \textbf{HK3}): for all
$F\in \mathrm{Diff}_{+}(S^{1})$ and $g\in G$,
$$
\beta_{g}\circ \mathrm{Ad}_{U(F)} \;=\; \mathrm{Ad}_{U(F)}\circ \beta_{g}.
$$
Equivalently, geometric covariance and the $G$-action commute on each local algebra.

Define the fixed-point subnet by
$$
\mathcal{A}^{G}(I)\;:=\;\mathcal{B}(I)^{G}
\;=\;
\{\,x\in \mathcal{B}(I)\mid \beta_{g}(x)=x\ \text{for all }g\in G\,\},
$$
and set
$$
\mathcal{A}^{G}(S^{1})\;:=\;\mathcal{B}(S^{1})^{G}.
$$

For a vertical inclusion $i:I\hookrightarrow I'$, define $\mathcal{A}^{G}(i)$ to be the (literal) inclusion
$$
\mathcal{B}(I)^{G}\hookrightarrow \mathcal{B}(I')^{G}.
$$

For a horizontal arrow $h=(F,I):I\to F(I)$, define
$$
\mathcal{A}^{G}(h)\;:=\;\mathrm{Ad}_{U(F)}\big|_{\mathcal{B}(I)^{G}}
:\mathcal{B}(I)^{G}\longrightarrow \mathcal{B}(F(I))^{G}.
$$

This is well-defined because $\mathrm{Ad}_{U(F)}$ commutes with the $G$-action: if $x\in \mathcal{B}(I)^G$, then for each $g\in G$,
$$
\beta_g\bigl(\mathrm{Ad}_{U(F)}(x)\bigr)=\mathrm{Ad}_{U(F)}\bigl(\beta_g(x)\bigr)=\mathrm{Ad}_{U(F)}(x),
$$
so $\mathrm{Ad}_{U(F)}(x)\in \mathcal{B}(F(I))^G$.

We now show the Haag-Kastler axioms.

\begin{itemize}
\item \textbf{HK1 (isotony) and HK2 (locality).}
Both are inherited from $\mathcal{B}$: for each arc $I$ we have $\mathcal{A}^{G}(I)=\mathcal{B}(I)^{G}\subseteq \mathcal{B}(I)$, and the inclusion maps are literal inclusions. Thus isotony holds. If $I$ and $J$ are disjoint arcs, then $[\mathcal{B}(I),\mathcal{B}(J)]=0$ by locality of $\mathcal{B}$, hence also $[\mathcal{A}^{G}(I),\mathcal{A}^{G}(J)]=0$.

\item \textbf{HK3 (functorial covariance).}
For a horizontal arrow $h=(F,I)$ we defined $\mathcal{A}^{G}(h)=\mathrm{Ad}_{U(F)}|_{\mathcal{B}(I)^{G}}$.
The commutation hypothesis $\beta_{g}\circ \mathrm{Ad}_{U(F)}=\mathrm{Ad}_{U(F)}\circ \beta_{g}$ implies that
$\mathrm{Ad}_{U(F)}$ preserves fixed points, so $\mathcal{A}^{G}(h)$ lands in $\mathcal{B}(F(I))^{G}$.
Compatibility with composition is the same calculation as in the untwisted net: if $h_{1}=(F_{1},I_{1}):I_{1}\to F_{1}(I_{1})$ and $h_{2}=(F_{2},F_{1}(I_{1})):F_{1}(I_{1})\to F_{2}F_{1}(I_{1})$, then
$$
\mathcal{A}^{G}(h_{2})\circ \mathcal{A}^{G}(h_{1})
=\mathrm{Ad}_{U(F_{2})}\circ \mathrm{Ad}_{U(F_{1})}
=\mathrm{Ad}_{U(F_{2}F_{1})}
=\mathcal{A}^{G}(F_{2}F_{1},I_{1}),
$$
and vertical functoriality is tautological for inclusions.

\item \textbf{HK4 (time-slice).}
As in the circle setup, we take the time-slice class to be empty in this example, so \textbf{HK4} is vacuous.

\item \textbf{HK5 (strong additivity; and additivity under inner continuity).}
By Xu \cite{xu2003strongadditivityconformalnets}, if $\mathcal{B}$ is strongly additive and a compact
group $G$ acts by vacuum-preserving internal symmetries, then the fixed-point subnet $\mathcal{B}^{G}$
is strongly additive. Since finite groups are compact, $\mathcal{A}^{G}$ is strongly additive.

If, in addition, one assumes the standard inner continuity property for a conformal net $N$,
namely that for every arc $I$,
$$
N(I)=\bigvee_{J\Subset I} N(J),
\qquad (J\Subset I \text{ meaning } \overline{J}\subset I),
$$
then strong additivity implies the usual open-cover additivity:

\begin{lemma}{Strong additivity + inner continuity $\implies$ additivity for open covers}
Let $N$ be a conformal net on $S^{1}$ that is strongly additive and inner continuous. Then for every arc
$I$ and every open cover $I=\bigcup_{\lambda} I_{\lambda}$ by arcs,
$$
N(I)=\bigvee_{\lambda} N(I_{\lambda}).
$$
\end{lemma}

\begin{proof}
Let $J,L\subset S^{1}$ be arcs with $J\cap L\neq\varnothing$ and $J\cup L$ an arc. Choose an
orientation-preserving parametrization $\varphi:J\cup L\to (a,d)\subset\mathbb{R}$ and write
$\varphi(J)=(a,b)$ and $\varphi(L)=(c,d)$ with $a<c<b<d$. Pick $p\in(c,b)$ and set
$I_{0}:=J\cup L$, $I_{1}:=\varphi^{-1}((a,p))$, and $I_{2}:=\varphi^{-1}((p,d))$.
Then $I_{0}\setminus\{\varphi^{-1}(p)\}=I_{1}\cup I_{2}$ is the decomposition into the two connected
components, so strong additivity gives $N(I_{0})=N(I_{1})\vee N(I_{2})$. Since $I_{1}\subset J$ and
$I_{2}\subset L$, isotony yields $N(I_{1})\subset N(J)$ and $N(I_{2})\subset N(L)$, hence
$N(I_{0})\subset N(J)\vee N(L)$. The reverse inclusion $N(J)\vee N(L)\subset N(I_{0})$ is isotony. This
proves $N(J\cup L)=N(J)\vee N(L)$.

Let $J_{1},\dots,J_{n}$ be arcs such that $J:=\bigcup_{k=1}^{n}J_{k}$ is an arc. Choose an
orientation-preserving parametrization $\psi:J\to (u,v)\subset\mathbb{R}$ and write
$\psi(J_{k})=(\ell_{k},r_{k})$. After reindexing, assume $\ell_{1}\le\cdots\le\ell_{n}$ and define
$J^{(1)}:=J_{1}$ and $J^{(m+1)}:=J^{(m)}\cup J_{m+1}$ for $1\le m\le n-1$. Since $\bigcup_{k=1}^{n}
(\ell_{k},r_{k})$ is an interval, one has $J^{(m)}\cap J_{m+1}\neq\varnothing$ for each $m$; otherwise
there would be a nonempty gap and the union would be disconnected. Hence each $J^{(m+1)}$ is an arc and
the two-arc case yields
$$
N(J^{(m+1)})=N(J^{(m)})\vee N(J_{m+1})\qquad(1\le m\le n-1).
$$
Iterating gives $N(J)=\bigvee_{k=1}^{n}N(J_{k})$.

Let $I$ be an arc and let $\{I_{\lambda}\}_{\lambda}$ be an open cover of $I$. Fix $J\Subset I$.
Compactness of $\overline{J}$ yields a finite subcover
$\overline{J}\subset \bigcup_{k=1}^{n} I_{\lambda_{k}}$. Choose an orientation-preserving
parametrization $\theta:J\to (s,t)\subset\mathbb{R}$. Then $\{\theta(I_{\lambda_{k}}\cap J)\}_{k=1}^{n}$
is a finite open cover of the compact interval $\overline{\theta(J)}=[s',t']$. From this finite cover one
can extract a finite subfamily $K_{1},\dots,K_{m}$ which still covers $[s',t']$ and satisfies
$K_{j}\cap K_{j+1}\neq\varnothing$ for all $j$. Pulling back along $\theta^{-1}$ gives arcs
$J_{1},\dots,J_{m}\subset J$ such that $\bigcup_{j=1}^{m}J_{j}=J$, each $J_{j}\subset I_{\lambda_{k(j)}}$
for some $k(j)$, and $J_{j}\cap J_{j+1}\neq\varnothing$. Therefore
$$
N(J)=\bigvee_{j=1}^{m} N(J_{j})\subset \bigvee_{k=1}^{n} N(I_{\lambda_{k}})\subset \bigvee_{\lambda} N(I_{\lambda}).
$$
Using inner continuity,
$$
N(I)=\bigvee_{J\Subset I} N(J)\subset \bigvee_{\lambda} N(I_{\lambda}).
$$
The reverse inclusion is isotony, hence $N(I)=\bigvee_{\lambda} N(I_{\lambda})$.
\end{proof}

If inner continuity is among the standing net assumptions, this lemma applies to $N=\mathcal{A}^{G}$ and
yields the open-cover additivity form of \textbf{HK5}. Without inner continuity, one writes \textbf{HK5}
for this example in the strong additivity form furnished by Xu.
\end{itemize}

\section{Conclusion and Preview of Part II}

\textbf{Summary.} Part I packages an AQFT input $\mathcal A$ as a pseudo double functor
$$
\mathcal F_{\mathcal A}\colon \mathbf{Mink}(M)\longrightarrow \mathbf{vNA},
$$
whose vertical component is the usual Haag-Kastler net $U\mapsto \mathcal A(U)$, while the horizontal and square-level assignments are designed so that the correspondence calculus is functorial rather than ad hoc. Concretely, the target pseudo double category $\mathbf{vNA}$ is chosen so that its vertical morphisms form a specified class $\mathbf{vN}_{\mathrm{vert}}$ of normal unital $*$-homomorphisms, its horizontal $1$-cells are von Neumann correspondences (bimodules), and its squares are bimodular intertwiners.

The key operator-algebraic input is the standard-form identity correspondence. For each von Neumann algebra $A$ we fix a Haagerup standard form $(H_A,\lambda_A,J_A,P_A)$, and define
$$
L^2(A)\coloneqq H_A
$$
as an $A$-$A$ correspondence, with left action given by $\lambda_A\colon A\to B(H_A)$ and right action
$$
\rho_{L^2(A)}(a)\coloneqq J_A\,\lambda_A(a^*)\,J_A,\quad a\in A.
$$

Two nonnegotiables from Part I are worth emphasizing, since they are exactly what makes the construction robust enough to support Part II.

First, the class $\mathbf{vN}_{\mathrm{vert}}$ of vertical morphisms is defined by a functoriality requirement: it consists of those normal unital $*$-homomorphisms $\varphi\colon A\to B$ for which the operator-algebraic framework provides a canonical standard-form map
$$
L^2(\varphi)\colon L^2(A)\longrightarrow L^2(B),
$$
functorial in $\varphi$ and compatible with the correspondence calculus (in particular with Connes fusion and the associated bimodule structures).
In this way, functoriality of $L^2$ is recorded as an explicit hypothesis on $\mathbf{vN}_{\mathrm{vert}}$, rather than being treated as an implicit assumption.

Second, the square assignment is typed in a way that genuinely uses spacetime commutativity: a commuting boundary
$$
j\circ h = h'\circ i
$$
in $\mathbf{Mink}(M)$ yields a square in $\mathbf{vNA}$ whose interior is the standard-form map
$$
L^2\!\bigl(\mathcal A(j)\bigr)\colon L^2\!\bigl(\mathcal A(\mathrm{dom}(j))\bigr)\longrightarrow L^2\!\bigl(\mathcal A(\mathrm{cod}(j))\bigr),
$$
and the induced commutativity
$$
\mathcal A(j)\circ \mathcal A(h)\;=\;\mathcal A(h')\circ \mathcal A(i)
$$
is exactly what guarantees that $L^2(\mathcal A(j))$ is bimodular for the restricted left and right actions arising from the horizontal assignments.
This is precisely how $\mathcal F_{\mathcal A}$ says more than the extracted net, while remaining true to standard operator-algebraic identities.

At this point it is reasonable to ask whether the double-categorical packaging is merely decorative, in the sense that it reformulates the net without producing new intrinsic consequences at the level of Haag-Kastler axioms.

Part II shifts the emphasis: rather than treating $\mathcal F_{\mathcal A}$ as a device that should, by itself, impose additional net-theoretic constraints, we use $\mathcal F_{\mathcal A}$ as a structured interface with the internal theory of von Neumann algebras.

\textbf{Preview of Part II.} For type I and type II von Neumann algebras, many standard structural operations are implemented by bimodules and bimodular maps; the square-level data in $\mathcal F_{\mathcal A}$ is precisely what records these implementations functorially.

Part II takes inspiration from Penneys' tracial $2$-categorical framework, and the first point to keep explicit is trace-dependence \cite{PenneysINI2017}.
In Part I, $L^2(A)$ is defined from a chosen Haagerup standard form and does not require any trace.

By contrast, in Penneys' setting the basic objects are pairs $(A,\mathrm{tr}_A)$, and the standard representation is the GNS Hilbert space $L^2(A,\mathrm{tr}_A)$ obtained by completing $A$ with inner product
$$
\langle a,b\rangle := \mathrm{tr}_A(b^*a).
$$
The right action and the modular conjugation $J$ are defined using traciality, and changing the trace changes the object.

Part II therefore aims to treat tracial AQFT as a refinement of the Part I target, not a change of notation.
When we work in Penneys' tracial setting \cite{PenneysINI2017}, objects are pairs $(A,\tau)$ and
$$
L^2 \hspace{0.2cm}\text{means}\hspace{0.2cm} L^2(A,\tau).
$$

When we introduce a vertical direction, we restrict the vertical $*$-homomorphisms to those that respect the chosen trace (and, in the semifinite case, the chosen tracial weight), in the same way that Part I restricts to $\mathbf{vN}_{\mathrm{vert}}$ to ensure functorial maps $L^2(\varphi)$.

This is the only point where the two frameworks can be confused, and we build it into the target to avoid ambiguity.

More precisely, we do not take Penneys' $2$-category as an additional ambient structure for this paper, nor do we rely on any external $W^*$-$2$-categorical formalism as input.
Rather, \cite{PenneysINI2017} serves as a guide to which trace-sensitive mechanisms should be internalized into our own double-categorical framework, and to how these mechanisms are best organized once traces are present.

Accordingly, Part II constructs a tracially decorated refinement of the Part I codomain, denoted
$$
\mathbf{vNA}_{\tau},
$$
whose objects are pairs $(A,\tau)$ where $\tau$ is a fixed faithful normal trace (in the finite case) or a fixed faithful normal semifinite tracial weight (in the semifinite case), whose vertical morphisms are normal unital $*$-homomorphisms compatible with the chosen tracial data, whose horizontal $1$-cells are correspondences formulated in this tracial/semifinite setting, and whose squares are bimodular intertwiners.
There is a forgetful pseudo double functor
$$
\mathbf{vNA}_{\tau}\longrightarrow \mathbf{vNA}
$$
which discards the tracial decoration and recovers the Part I correspondence calculus.

The point of introducing $\mathbf{vNA}_{\tau}$ is not to impose new axioms on AQFT nets, but to make standard type I/type II operations expressible at the same level of functoriality already present in $\mathcal F_{\mathcal A}$.
In particular, Part II builds into $\mathbf{vNA}_{\tau}$ the basic corner and amplification operations that are ubiquitous in the structure theory of type II factors, and organizes them as canonical constructions compatible with whiskering and pasting of squares.

Under the standing trace-compatibility hypotheses, a tracial AQFT input is then an AQFT input $\mathcal A$ together with a choice of tracial data $\tau_U$ on each local algebra $\mathcal A(U)$ such that the morphisms associated to admissible embeddings preserve the chosen tracial data (in the finite case) or preserve the chosen tracial weight (in the semifinite case).
From such input, Part II constructs a refined pseudo double functor
$$
\mathcal F_{(\mathcal A,\tau)}\colon \mathbf{Mink}(M)\longrightarrow \mathbf{vNA}_{\tau}
$$
whose composite with the forgetful functor $\mathbf{vNA}_{\tau}\to \mathbf{vNA}$ is canonically isomorphic to the original $\mathcal F_{\mathcal A}$.
Thus the extracted net remains $U\mapsto \mathcal A(U)$, while the horizontal and square-level assignments land in a codomain where trace-sensitive constructions are functorial by design.

A guiding example of the kind of mechanism we internalize is Roberts' $2\times 2$ trick, which identifies spaces of bimodule maps as corners in endomorphism algebras of direct sums and thereby makes corner manipulations compatible with composition \cite{PenneysINI2017}.
In the present framework, the role of this principle is to motivate and justify the inclusion of direct-sum and corner operations in the square-level calculus of $\mathbf{vNA}_{\tau}$, so that standard amplification and corner arguments can be expressed in terms of whiskering and pasting rather than as external algebraic maneuvers.

Finally, Part II uses this refined codomain to formulate, and to test in representative example classes, how familiar type II themes can be expressed in the functorial AQFT language.
The emphasis is on a technically correct dictionary and on functorial constructions, rather than on asserting classification results from the Haag-Kastler axioms alone: properties such as McDuff-type stability, Cartan-type structure, primeness, and the $\mathrm{II}_1/\mathrm{II}_{\infty}$ semifinite contrast are treated as trace-sensitive conditions on local algebras and their correspondence calculus, and Part II develops the formalism needed to state and transport these conditions coherently along the spacetime structure encoded in $\mathbf{Mink}(M)$.

\printbibliography

\end{document}